\documentclass[
12pt]
{amsart}

\usepackage{hyperref}
\usepackage{cite}

\usepackage[a4paper, margin=2.5cm]{geometry}
\usepackage{graphicx}
\usepackage{xcolor}
\usepackage[toc,page]{appendix}
\usepackage[normalem]{ulem}
\usepackage{comment}

\usepackage{esint}
\usepackage{amsmath}   
\usepackage{amssymb}   
\usepackage{amsthm}    
\usepackage{amsfonts}  
\usepackage{bbm}
\usepackage{mathrsfs}  
\usepackage{esint}
\usepackage{lipsum}
\usepackage{mathtools}

\usepackage[shortlabels, inline]{enumitem}

\theoremstyle{plain}
\newtheorem{theorem}{Theorem}[section]
\newtheorem{corollary}[theorem]{Corollary}
\newtheorem{lemma}[theorem]{Lemma}
\newtheorem{proposition}[theorem]{Proposition}
\newtheorem{definition}[theorem]{Definition}

\theoremstyle{remark}
\newtheorem{remark}[theorem]{Remark}
\newtheorem{example}[theorem]{Example}

\newtheorem{claim}[theorem]{Claim}

\theoremstyle{definition}
\newtheorem*{acknowledgement}{Acknowledgement}

\DeclarePairedDelimiterX{\hsp}[2]{\lparen}{\rparen}{#1 \mid #2}
\DeclarePairedDelimiterX{\pairing}[2]{\langle}{\rangle}{#1 \mid #2}

\newcommand{\wick}[1]{\mathopen{:}#1\mathclose{:}}

\newcommand{\set}[1]{\left\{#1\right\}}							

\newcommand{\seq}[1]{\left(#1\right)}								

\newcommand{\tparen}[1]{\big({#1}\big)}							
\newcommand{\paren}[1]{\left(#1\right)}							

\newcommand{\tbraket}[1]{\big[#1\big]}							

\newcommand{\scalar}[2]{\left\langle #1 \,\middle |\, #2\right\rangle}

\newcommand{\comma}{\,\,\mathrm{,}\;\,}

\newcommand{\fstop}{\,\,\mathrm{.}}

\newcommand{\emparg}{{\,\cdot\,}}								
\renewcommand{\paragraph}[1]{\smallskip\noindent\emph{#1.}\;\,}

\def\N{{\mathbb N}}    
\def\R{{\mathbb R}}
\def\Q{{\mathbb Q}}
\def\E{{\mathbb E}}
\def\P{{\mathbb P}}

\def\M{{M}}
\def\g{{g}}
\newcommand{\vol}{\mathsf{vol}}

\DeclareMathOperator{\Ric}{\mathrm{Ric}}

\def\C{\mathcal{C}}

\def\eps{{\varepsilon}}

\newcommand{\masha}[1]
{{\color{red} Masha says: #1}}

\newcommand{\eva}[1]
{{\color{blue} Eva says: #1}}

\begin{document}

	\title[Semiclassical limit]{
		SEMICLASSICAL LIMIT OF POLYAKOV-LIOUVILLE MEASURE 
		and Q-Curvature Uniformization
		ON
		EVEN-DIMENSIONAL MANIFOLDS
	}
	
	\author[Gordina]{Maria Gordina}
	\address{Department of Mathematics\\
		University of Rochester\\
		Rochester, NY 14627,  U.S.A.}
	\email{maria.gordina@rochester.edu}

	\author[Kopfer]{Eva Kopfer}
	\address{ Institute for Applied Mathematics \\
		University of Bonn\\
		Germany 
	}
	\email{eva.kopfer@iam.uni-bonn.de}
	
	\author[Sturm]{Karl-Theodor Sturm}
	\address{ 
		Institute for Applied Mathematics\\
		University of Bonn\\
		Germany 
	}
	\email{sturm@uni-bonn.de}
	
	

	\date{}

	\begin{abstract}
		We study the semiclassical limit of the Polyakov-Liouville measure $\boldsymbol{\nu}_\gamma$, which is a non-Gaussian measure on $H^{-\eps}(M)$ that has recently been extended from Riemann surfaces to general Riemannian manifolds $(M,g)$ of even dimension. We show that under an appropriate rescaling in the semiclassical limit as $\gamma\to0$, the normalized Polyakov-Liouville measure $\Q_\gamma$ concentrates on the unique smooth weight $u$ for which the conformal metric $e^{2u}g$ on $M$ has constant $Q$-curvature.
	\end{abstract}

	\maketitle
	
	\tableofcontents

	\section{Introduction} 

Over the past two decades, remarkable progress has been achieved in the probabilistic approach to conformal field theory, most notably through the work of David, Kupiainen, Rhodes and Vargas in \cite{DKRV16}. A central object in this theory is the Polyakov–Liouville measure $\boldsymbol{\nu}_\gamma(h)$, a non-Gaussian probability measure on the space of distribution-valued fields
$H^{-\eps}(M)$, where $(M,g)$ is a Riemann surface.

In \cite{lacoin2017semiclassical} the authors study the semiclassical limit of Liouville field theory, 
namely the regime in which $\gamma\to 0$ while $\Lambda:=m\gamma^2$ remains fixed so that $m\to\infty$. They show that for flat reference metrics $g$ the rescaled fields $\frac\gamma2 h$ converge in law to the solution of the classical Liouville equation. They further  determine the leading order fluctuations around this limiting hyperbolic geometry, proving that they are  given by a massive Gaussian free field with mass proportional to $\Lambda$, which coincides with the absolute value of the constant negative curvature.

The main goal of this paper is to extend the semiclassical limit of Liouville field theory to higher even-dimensional manifolds. In particular, we show that in the limit $\gamma\to 0$ the Polyakov-Liouville measure ${\boldsymbol{\nu}}_{\gamma}$ concentrates around conformal metrics of constant $Q$-curvature, generalizing the corresponding two-dimensional result in \cite{lacoin2017semiclassical}.
That is, we study asymptotic behavior of ${\boldsymbol{\nu}}_{\gamma}$ in the limit $\gamma\to0$ in the semiclassical regime.

Recently, the authors in \cite{Del24} proposed  an ansatz for a conformal field theory on compact manifolds of arbitrary even dimension. The approach employs Branson's $Q$-curvature and the Graham--Jenne--Mason--Sparling (GJMS) operator of maximal order.
On an even $n$-dimensional manifold $M$ the GJMS operator of maximal order 
\begin{equation*}
	\mathsf p=a_n(-\Delta)^{n/2} + \text{low order terms},
\end{equation*}
where $a_n=\frac{2}{(4\pi)^{n/2}\Gamma(n/2)}$,
has been first defined in \cite{GJMS92} and plays the role of a conformally covariant power of the Laplacian.

Let $g$ be a Riemannian metric on $M$ and let $\hat g=e^{2u}g$ be the conformally related metric for some smooth function $u\colon M\to\R$. We denote by $\mathsf p_g,\mathsf p_{\hat g}$  the corresponding GJMS operators,  and by and $Q_g,Q_{\hat g}$ the  $Q$-curvatures  associated with the metrics $g$ and $\hat g$.   It was shown in \cite{Bra85} that these quantities satisfy the following conformal transformation rule

\begin{align}\label{trafo-Q}
	e^{nu}Q_{\hat g}=Q_g+\frac1{a_n}\mathsf{p}_gu.
\end{align}
Similarly to  the uniformization theorem for compact surfaces, one may ask for a conformal change $\hat g=e^{2 u}g$ so that the new metric has constant $Q$-curvature $-\Lambda$ for some $\Lambda \in \R$, i.e.
\[
Q_{e^{2u}g}=-\Lambda.
\]
By the transformation rule for $\mathsf p$ and $Q$ this problem is equivalent to finding a solution of
\begin{align}\label{eq: liouville0}
	\frac1{a_n}	\mathsf pu+Q+\Lambda e^{nu}=0,
\end{align}
where $\Lambda\in\R$. Integrating equation \eqref{eq: liouville0} over $M$ yields a necessary condition for the existence of a solution, namely that the total $Q$-curvature $Q(M)=\int Q\, d\vol$ satisfies
\begin{align*}
	Q(M)=-\Lambda \int e^{nu}\, d\vol.
\end{align*}

It was shown in \cite{Ndiaye} that if the GJMS operator $\mathsf p$ has one-dimensional kernel and the total $Q$-curvature satisfies
\begin{align*}
	Q(M)\neq k(n-1)!\omega_n,\quad k=1,2,\ldots,
\end{align*}
then there exists a smooth function $u\in C^\infty(M)$ such that the conformal metric $\hat g=e^{2u}g$ has constant $Q$-curvature $-\Lambda$. 

Based on this, we show in Corollary \ref{cor: liouville}, that if the total $Q$-curvature is non-positive, then 	for every $\kappa\in\R$ there exists a unique smooth function $u=u_\kappa\in C^\infty(M)$ such that 
\begin{align}\label{const-Q}
	e^{2u}g \text{ has constant $Q$-curvature and }
	\frac1{|M|}\int_M u\,d\vol=\kappa.
\end{align}
It follows that $u_\kappa=u_0+\kappa$. Then the $Q$-curvature of $(M, e^{2u_\kappa}g)$ is $-\Lambda_\kappa=-e^{-n\kappa}\Lambda_0$, where  $\Lambda_\kappa:=-Q(M)\,\left(\int_M e^{nu_\kappa}d\vol\right)^{-1}$.


Recall that the Liouville field theory is a probabilistic model that 
studies the \emph{Polyakov--Liouville measure} formally given by 
\begin{equation*}
	\exp(-S(h)) dh
\end{equation*}
with (non-existing) Lebesgue measure $dh$ on the set of fields 
and action functional
\begin{equation*}
	S(h) =   \int_M\Big( \frac12h\mathsf ph +\Theta\,Q h + {m} e^{\gamma h}\Big) d \vol.
\end{equation*}
Here,
$m,\Theta,\gamma$ are parameters,
where  $\gamma>0$ is the coupling constant of the exponential interaction, $\Theta\in\R$ is the coupling to the background $Q$-curvature and $m>0$ is the cosmological constant controlling the total mass.

In two dimensions, this theory plays an important role in bosonic string theory and in Liouville quantum gravity, which describes random conformal metrics, see \cite{Pol81a}.

From a mathematical perspective, the model is challenging because the field $h$ is not a function but a distribution, so that, in particular, $e^{\gamma h}$ is not well-defined. This difficulty is resolved using the theory of Gaussian multiplicative chaos introduced in \cite{Kah85}, leading to the \emph{Liouville quantum gravity measure} $\mu^{\gamma h}$. See e.g. \cite{DS11, RhoVar14, Sha16, Ber17} for the construction of the Liouville measure and \cite{DRV16,DKRV16,GRV19} for the application in Liouville quantum gravity. 

A rigorous construction of Liouville quantum gravity in higher even dimensions was first proposed in the physics literature \cite{LevyOz}, and then mathematically rigorous constructions have been developed for spheres in \cite{Cercle} and for admissible manifolds in \cite{Del24}.

In this paper, we follow the approach in \cite{Del24} and assume that the manifold $M$ is \emph{admissible}, meaning it is closed, even-dimensional and 
\begin{align*}
	\mathsf p>0 \text{ on } \mathring H,
\end{align*}
where $\mathring H$ denotes the \emph{grounded} $L^2$-space, i.e.
\begin{align*}
	\mathring H=\{\phi\in L^2(M): \langle \phi\rangle=0\}, \quad \langle \phi\rangle=\frac1{|M|}\int_M  \phi(x)\, d\vol(x).
\end{align*}
As shown in \cite{Del24}, under this assumption there exists a pseudo-inverse $\mathsf k_0$ whose  kernel $k_{0}(x,y)$,
satisfies the logarithmic divergence
\begin{equation*}
	\Big|k_{0}(x,y)-\log\frac1{d(x,y)}\Big|\le C\,.
\end{equation*}
The \emph{co-polyharmonic Gaussian field} $h_0$ defined in \cite{Del24} is the centered Gaussian field with covariance kernel
\[
\E\big[h_0(x)\,h_0(y)\big]=k_{0}(x,y).
\]
This field generalizes the classical Gaussian free field to higher even dimensions.
As in two dimensions, $h_0$ is not a classical functions on $M$ but rather a random distribution in $\mathfrak{D}'$, the dual space of~$\mathfrak{D}=\C^\infty(M)$. We write $\scalar{\cdot}{\cdot}$ for the dual pairing between these spaces.

We consider the \emph{augmented field} $h$, which is the centered Gaussian field with 
covariance kernel
\[
\E\big[h(x)\,h(y)\big]=k(x,y), \quad k(x,y):=k_{0}(x,y)+\sigma^2.
\]

Formally, the law $\P$ of $h$ is given by
\begin{align*}
	d\P(h)\propto\exp\left(-\frac12\left( \scalar{h}{\mathsf p h} + \frac1{\sigma^2}\langle h\rangle^2\right)\right)dh.
\end{align*}
Equivalently, since $\mathsf{p}$ is self-adjoint and annihilates constants, the augmentation amounts to adding an independent $\mathcal N(0,\sigma^2)$-distributed random variable to the co-polyharmonic field $h_0$.

The covariance function $k(x,y)$ retains the same logarithmic singularity as $k_0(x,y)$, but the additive constant is chosen so that $k(x,y)\geq0$ pointwise. More precisely, we will choose $\sigma\in\R$ such that the  integral kernel for  $\sqrt{{\sf k_0}}$ satisfies
$g_0(x,y)\ge-{|M|}^{-1/2}\sigma$, cf.~Lemma \ref{est-sqrt-k}.
%
The square root of $\mathsf k_0$ exists since the latter is self-adjoint and positive definite. The positivity of the kernel of $\sqrt{\sf k}:=\sqrt{\sf k_0}+{|M|}^{-1/2}\sigma$ entails positivity of the kernel $k(x,y):=k_0(x,y)+\sigma^2$,
and it yields a more regular cut-off approximation, which we now briefly explain.

Formally, the field $h$ can be represented as
\begin{align*}
	h(x)=\sqrt{\mathsf k}W(x),
\end{align*}
where $W$ denotes white noise on $M$. For ${\ell}\in\N$, the \emph{white-noise truncation} $h_{\ell}$ of $h$ is defined by removing the singularity near $x$ :
\begin{align*}
	h_{\ell}(x)
	=\sqrt{\mathsf k}(1_{M\setminus B_{1/l(x)}}W)(x).
\end{align*}
The field $h_{\ell}$ is a centered continuous Gaussian field on $M$. Moreover, for $m>\ell$, the increments $(h_m-h_{\ell})(x)$ are independent and have nonnegative correlations:
\begin{align*}
	\E[(h_m-h_{\ell})(x)(h_m-h_{\ell})(y)]\geq0, \quad \forall x,y\in M, m>\ell.
\end{align*} 
See Section \ref{sec: aug} for the precise construction.

For $0<\gamma<\sqrt{2n}$ we denote by $\mu^{\gamma h}$ the \emph{Liouville quantum gravity measure} associated to the augmented field $h$. For $\kappa\in\R$, we consider $\P_{{n\kappa/\gamma}}$, which is the distribution of $h+n\kappa/\gamma$, with the convention that $\P_0=\P$. Formally it holds that
\begin{align*}
	d\P_{{n\kappa/\gamma}}(h)\propto\exp\left(-\frac12\left( \scalar{h}{\mathsf p h} + \frac1{\sigma^2}\langle h-\frac{n}\gamma\kappa\rangle^2\right)\right)dh.
\end{align*}
We define the \emph{Polyakov--Liouville measure} by
\begin{equation*}
	d{\boldsymbol{\nu}}_{\kappa,\gamma}(h):= \exp\Big(-\Theta\scalar{h}{Q}- m \,\mu^{\gamma h}(M)\Big)\,d\P_{{n\kappa/\gamma}}(h).
\end{equation*}
In order to study the semiclassical limit, 
we choose 
$$\Theta=\frac{na_n}{\gamma},  \quad m=\frac{na_n\Lambda_\kappa}{\gamma^2}.$$
We show convergence of the rescaled augmented field $\frac{\gamma}nh$ to the unique solution $u_\kappa$ of \eqref{const-Q} under
the normalized Polyakov-Liouville measure $\Q_{\kappa,\gamma}$, which is defined as
\begin{align*}
	\Q_{\kappa,\gamma}=\frac1{Z_{\kappa,\gamma}}\boldsymbol{\nu}_{\kappa,\gamma}.
\end{align*}
Moreover, we show that the first order fluctuation is given by an \emph{augmented massive field} $\tilde h$ on $(M,\hat g)$, where $\hat g=e^{2u_\kappa}g$, with mass $na_n\Lambda_\kappa$, which is a centered Gaussian field with covariance given by
\begin{align*}
	\E\left[\scalar{\tilde h}{\phi}\scalar{\tilde h}{\eta}\right]=\scalar{\phi}{\mathring {\mathsf p}_{na_n\Lambda_\kappa}^{-1}\eta} + \sigma^2\scalar{\phi}{1}\scalar{\eta}{1}, \qquad \forall \phi,\eta\in {\mathfrak D}, 
\end{align*} 
where 
\begin{align*}
	\mathring {\mathsf p}_{na_n\Lambda_\kappa}^{-1}(\phi)=(\mathsf p+ na_n\Lambda_\kappa)^{-1}(\phi-\langle \phi\rangle).
\end{align*}

Our main result reads as follows.

\begin{theorem}\label{thm: main}
	Let $(M,g)$ be an admissible manifold with $Q(M)\le0$. For $\kappa\in\R$, let $u_\kappa$ be the unique smooth solution to \eqref{const-Q}. 
	\begin{enumerate}[$(i)$]
		\item Under $\Q_{\kappa,\gamma}$, the rescaled augmented fields $\frac\gamma n h$ converge in law as $\gamma\to0$ 
		towards $u_\kappa$. 
		\item Under $\Q_{\kappa,\gamma}$, the fields $h-\frac n\gamma u_\kappa$ converge in law towards an augmented massive field $\tilde h$ on $(M,\hat g)$, where $\hat g=e^{2u_\kappa}g$, with mass $na_n\Lambda_\kappa$. 
	\end{enumerate}
	In both cases, the respective fields are regarded as random variables with values in $H^{-1}(M)$. In the borderline case $Q(M)=0$, the ``massive’’ field $\tilde h$ in $(ii)$ actually is mass-less.
\end{theorem}

\begin{remark}
	The result is consistent with the Laplace principle at speed $n^2/\gamma^2$. 
	This scaling naturally arises  by considering  the push forward under the map
	$h\mapsto \frac \gamma n h$ of
	the measure
	formally given as  
	\begin{equation*}
		d{\boldsymbol{\nu}}_{\kappa,\gamma}(h)= \exp\left(-\tilde S(h)\right)\, dh
	\end{equation*}
	with action functional  
	\[
	\tilde S(h)=\int_\M\left[\frac1{2}h \mathsf p h+\frac1{2\sigma^2}\langle h-\frac n\gamma \kappa\rangle^2 +\frac{na_n}\gamma Q h +\frac{na_n}{\gamma^2} \Lambda_\kappa e^{n h }\right]\, d\vol\,.
	\]
	After this re-scaling. 
	\begin{equation*}
		d{\boldsymbol{\nu}}^{\sf sc}_{\kappa,\gamma}(h)= \exp\left(-\tilde S\left(\frac n\gamma h\right)\right)\, dh =  \exp\left(-\frac{n^2}{\gamma^2}\hat S( h)\right)\, dh 
	\end{equation*}
	with action functional  
	\begin{align*}
		\hat S(h)=
		\int_\M\left[\frac12h \mathsf p h+\frac1{2\sigma^2}\langle h-\kappa\rangle^2 +a_nQ h +\frac{a_n}n \Lambda_\kappa e^{n h }\right]\, d\vol\,.
	\end{align*}
	The Laplace principle formally then says  that for a measurable set $A$
	\begin{align*}
		\lim_{\gamma\to0} \frac{\gamma^2}{n^2}\log\int_A\exp\left(-\frac{n^2}{\gamma^2}\hat S(h)\right)\, dh=-\inf_{h\in A} \hat S(h),
	\end{align*}
	Moreover,
	\begin{align*}
		\inf_{h\in A}\hat S(h)=\inf_{h\in A:\, \langle h\rangle =\kappa}\hat S(h).
	\end{align*}
	Thus the Euler-Lagrange equation (together with the constraint for the mean value) reads as
	\begin{align*}
		\frac1{a_n}\mathsf pu+Q+\Lambda_\kappa e^{nu}=0,\qquad  \langle u\rangle =\kappa.
	\end{align*}
	According to the transformation formula \eqref{trafo-Q}, this yields
	\begin{align*}
		Q_{e^{2u}g}=-\Lambda_\kappa,\qquad  \langle u\rangle =\kappa,
	\end{align*}
	in consistency with \eqref{const-Q}.
\end{remark}

\begin{remark}
	This work significantly extends the results by \cite{lacoin2017semiclassical} to higher even dimensions and non-flat geometries. Although the authors remark that their approach may extend to higher even dimensions using the GJMS operator at least in flat spaces \cite[p. 5]{lacoin2017semiclassical}, this extension is far from straightforward. For instance, in two dimensions, the white noise approximation 
	\begin{align*}
		h_\eps(x)\propto\int_{\eps^2}^\infty \int p_{r/2}(x,y)W(dr,dy)
	\end{align*}
	is well defined,
	where $W$ is space-time white noise. In higher dimensions, however, no direct analogue of this construction is available. In particular, the parabolic kernel associated with the operator $\mathsf p$ generally changes sign. This motivates why we introduce the white noise truncation $(h_{\ell})_{l\in\N}$. Furthermore, the paper \cite{lacoin2017semiclassical} studies the Gaussian free field on the disc with boundary condition, whereas the present work considers compact manifolds without boundaries. Even in $n=2$, our approach extends previous work since it is intrinsic and it applies to surfaces of arbitrary genus $\mathfrak g\ge1$.

\end{remark}

\begin{remark}
	
	More recently, semiclassical limits have been studied in boundary Liouville conformal field theory \cite{Cer25}. This work is subject to two-dimensional conformal invariance and boundary conditions, and the limiting behavior is described in terms of classical Liouville equations on domains with boundary. In contrast to \cite{lacoin2017semiclassical} and \cite{Cer25}, the present paper develops a framework valid on general compact Riemannian manifolds of even dimensions without boundary. The approach is purely real-analytic and operator-theoretic and does not rely on the complex-analytic structure or explicit conformal parametrizations available in two dimensions.
	
\end{remark}

\begin{remark}
	The role of the conformal radius requires a more subtle interpretation in dimensions $n>2$. To this end, let 
	\begin{align*}
		r(x)=\limsup_{y\to x}\left[k(x,y)-\log \frac1{d(x,y)}\right]\quad\forall x\in M.
	\end{align*}
	The \emph{adjusted Liouville quantum gravity measure} in \cite{Del24} is defined by 
	\begin{align*}
		d	\bar\mu^{\gamma h}(x)=e^{\frac{\gamma^2}2r(x)}d\mu^{\gamma h}(x).
	\end{align*}
	This additional normalization term $r(x)$ corresponds to the logarithm of the conformal radius $\log C(x,\mathbf U)$ in \cite{lacoin2017semiclassical}. Since $\bar\mu^{\gamma h}$ differs from $\mu^{\gamma h}$ only by a deterministic multiplicative weight of order $\gamma^2$, this modification does not affect the semiclassical limit as $\gamma\to0$. For this reason,  we simply consider $\mu^{\gamma h}$.
\end{remark}

We denote by $\Q_{\kappa,\gamma}^{\mathrm{sc}}$ the distribution of the rescaled fields 
$$\hbar:=\frac\gamma n h$$
under the normalized Polyakov-Liouville measure $\Q_{\kappa,\gamma}$. 
This way, indeed, we then prove that the rescaled Polyakov-Liouville measures concentrate on the set of conformal weights which lead to manifolds of constant $Q$-curvature. Indeed, for any  $\Lambda\ge0$ we can choose $\kappa\in\R$ such that the above random construction in the semiclassical limit leads to  manifolds of constant $Q$-curvature equal to $-\Lambda$.

\begin{corollary} Let $\Q_{\kappa,\gamma}^{\mathrm{sc}}$ denote the distribution of the rescaled fields 
	$\hbar:=\frac\gamma n h$ under the normalized Polyakov-Liouville measure $\Q_{\kappa,\gamma}$. Then under the assumptions of Theorem \ref{thm: main} as $\gamma\to0$,
	$$\Q_{\kappa,\gamma}^{\mathrm{sc}} \ \to \ \delta_u\quad\textrm{weakly as measures on }H^{-1}(M),$$
	where $u=u_\kappa$ denotes the unique smooth function on $M$
	with
	$\langle u\rangle =\kappa$ such that
	$e^{2u}g$ has constant $Q$-curvature.
\end{corollary}

\begin{remark} One easily verifies that the distribution of $\hbar$ under $\Q_{\kappa,\gamma}^{\sf sc}$ coincides with the distribution of $\hbar+\kappa$ under $\Q_{0,\gamma}^{\sf sc}$, see Lemma \ref{lemma: Qkappa}. Furthermore, as already observed
	$u_\kappa=u_0+\kappa$. Thus the convergence result for general $\kappa\in\R$ is an immediate consequence of the  convergence result for $\kappa=0$. 
\end{remark}
\subsection*{Organization of the paper}
In Section \ref{sec: co} we introduce the Graham--Jenne--Mason--Sparling operator of maximal order on admissible manifolds. In Section \ref{sec: liou} we show existence and uniqueness of solutions to the Liouville equation. Section \ref{sec: aug} contains the definition of the augmented field and its truncations, while Section \ref{sec: main} is devoted to the proof of Theorem \ref{thm: main}.

\begin{acknowledgement}
	MG gratefully acknowledges funding by NSF Grant DMS-2246549 and a Bonn Research Chair (the Hausdorff Center for Mathematics, Germany). 
	
	EK and KTS gratefully acknowledge funding by the Deutsche Forschungsgemeinschaft through the project `Random Riemannian Geometry' within the SPP 2265 \emph{Random Geometric Systems}, 
	through the Hausdorff Center for Mathematics (project ID 390685813), and through  project B03  within the CRC 1060 (project ID 211504053).
	
	Data sharing not applicable to this article as no datasets were generated or analyzed during the current study.
\end{acknowledgement}

\section{GJMS operators and admissible manifolds}\label{sec: co}
For a given manifold $(M,g)$ of even dimension $n$, we consider the (normalized) \emph{Graham--Jenne--Mason--Sparling operator  of maximal order}
\begin{equation*}
	\mathsf{p}\coloneqq a_n\,(-\Delta)^{n/2} + \text{low order terms}, \quad a_n=\frac{2}{\Gamma(n/2)(4\pi)^{n/2}}.
\end{equation*}
This differential operator~$\mathsf{p}$  has been first defined in \cite{GJMS92} and plays the role of a conformally covariant power of the Laplacian, i.e. if $\hat g=e^{2u}g$ for some $u\in C^\infty(M)$, then
\begin{align}\label{eq: gjms}
	\mathsf{p}_{\hat g}=e^{-nu}\mathsf{p}_{g}.
\end{align}
It has been shown in \cite{GZ03} that $\mathsf{p}$ is self-adjoint and annihilates constants.
For $n=2$, in our sign convention the operator~$\mathsf{p}$ is just $-\frac1{2\pi}\Delta$, and for $n=4$ it is the \emph{Paneitz operator} \cite{Paneitz}.

We call a compact even-dimensional Riemannian manifold~$(M,g)$ \emph{admissible} if
$$\mathsf p>0 \text{ on }\mathring H,$$
where
$\mathring H= \big\{\phi\in L^2(\M): \, \langle\phi\rangle=0\big\}$.

Admissibility is a conformal invariant.
All compact, non-negatively curved Einstein manifolds are admissible, and so are all compact hyperbolic manifolds with spectral gap $\lambda_1>\frac{n(n-2)}4$.
Of course, all compact 2-dimensional Riemannian manifolds are admissible.

One of the main results in \cite{Del24} states that for every admissible manifold, the inverse of~$\mathsf p$ on~$\mathring H$ has an integral kernel $k_0$ which annihilates constants and has precise logarithmic divergence \cite[Theorem 2.19]{Del24}
\begin{equation}\label{eq: log}
	\Big|k_0(x,y)-\log\frac1{d(x,y)}\Big|\le C\,.
\end{equation}

\subsubsection*{Sobolev spaces and pairings}
For  $s \in \mathbb{R}_+$, we define the usual \emph{Sobolev spaces}  ${\mathcal H}^{s}\coloneqq  {(1-\Delta)}^{-\frac{s}{2}} L^2$ with norm
$\|\phi\|_{{\mathcal H}^s}\coloneqq \| (1-\Delta)^{\frac s2}\phi\|_{L^2}$, and ${\mathcal H}^{-s}$ as the completion of $L^2$ w.r.t.~the norm
$\|\phi\|_{{\mathcal H}^s}\coloneqq \| (1-\Delta)^{-\frac s2}\phi\|_{L^2}$
such that formally ${\mathcal H}^{-s}\coloneqq  {(1-\Delta)}^{\frac{s}{2}} L^2$.


We define
$$H^{s}\coloneqq  {(1+\mathsf p)}^{-\frac{s}{n}} L^2, \qquad
\|\phi\|_{H^s}\coloneqq \| (1+\mathsf p)^{\frac s{n}}\phi\|_{L^2}$$
and set $H^{-s}$ as the completion of $L^2$ w.r.t.~the norm
$\|\phi\|_{H^{-s}}\coloneqq \| (1+\mathsf p)^{-\frac s{n}}\phi\|_{L^2}$.
Moreover,
we define the  \emph{grounded Sobolev spaces}
$$\mathring H^{s}\coloneqq  {\mathsf p}^{-\frac{s}{n}} \mathring H, \qquad
\|\phi\|_{\mathring H^s}\coloneqq \| {\mathsf p}^{\frac s{n}}\phi\|_{\mathring H}$$
and define $\mathring H^{-s}$ as the completion of $\mathring H$ w.r.t.~the norm
$\|\phi\|_{\mathring H^{-s}}\coloneqq \| {\mathsf p}^{-\frac s{n}}\phi\|_{\mathring H}$.

For every~$s\in\R$, the spaces $\mathcal H^s$ and $H^s$ coincide as sets and the respective norms are bi-Lipschitz equivalent to each other.

The operator $\mathsf{p}$ with domain $H^n$ has discrete spectrum
$$\mathrm{spec}(\mathsf{p})= \set{\nu_j}_{j\in \mathbb{N}_0},$$
indexed with multiplicities, satisfying~$\nu_j \geq 0$ for all~$j$, and~$\nu_0=0$ with multiplicity~$1$.
The corresponding family of eigenfunctions~$\seq{\psi_j}_{j\in\mathbb{N}_0}$ forms an orthonormal basis of~$L^2$.
Further, the spectrum of $\mathsf p$ satisfies the Weyl asymptotic, i.e.
\begin{equation}\label{eq: weyl}
	\nu_{j} = c j + O(j^{1-1/n})\comma \qquad j \to \infty\fstop
\end{equation}

For every~$r\in\R$, the bounded operator $\mathsf{p}\colon \mathring{H}^{n+r}\to \mathring{H}^r$ has bounded inverse $\mathsf{k}_0\colon \mathring{H}^r\to\mathring{H}^{n+r}$.
For $r=0$, the operator $\mathsf{k}_0\colon \mathring{L}^2\to \mathring{L}^2$
admits a unique non-relabeled extension $\mathsf{k}_0\colon L^2\to \mathring{L}^2$, vanishing on constants and satisfying~$\mathsf{k}_0\mathsf{p}=\pi$ on~$L^2$, where
\[
\pi(\phi)\coloneqq \phi-\langle \phi\rangle
\ .
\]

This extension is an integral operator on~$L^2$ with symmetric kernel
\begin{align*}
	k_0(x,y)=  \sum_{j =1}^\infty \frac{\psi_j(x)\, \psi_j(y)}{\nu_j} \comma \qquad x,y\in M\comma
\end{align*}
where the convergence of the series is understood in~$L^2\otimes L^2$. See Lemma 2.15 in \cite{Del24} for these facts.

We denote the associated bilinear forms for the Hilbert spaces $H^{n/2}$ and $\mathring H^{-n/2}$ by
\begin{align*}
	\mathfrak{p}(\phi,\eta) &\coloneqq \int \sqrt{\mathsf{p}} \phi\, \sqrt{\mathsf{p}} \eta \, d \vol\ , \qquad & \phi,\, \eta \in H^{n/2} \comma\\
	\mathcal{K}_0(\phi,\eta) &\coloneqq \int \sqrt{\mathsf{k}_0} \phi\, \sqrt{\mathsf{k}_0} \eta \, d\vol\ , \qquad & \phi,\, \eta \in \mathring H^{-n/2} \fstop
\end{align*}

Restricted to the space $\mathring H$, the bilinear form 
$\mathcal{K}_{0}$ is given by
\begin{align}
	\mathcal{K}_0(\phi,\eta)&= \scalar{\mathsf{k}_0\phi}{\eta}_{L^2}
	=\iint \phi(x)\, k_0(x,y)\, \eta(y)\, d\vol(x) \, d\vol(y) \,.
\end{align}
Observe that the right-hand side here is also well defined for ungrounded $\phi,\eta\in L^2$ which allows us to consider 
$\mathcal{K}_{0}$  also  as a bilinear form on $L^2$ with
$\mathcal{K}_0(\phi+C,\eta+C')=\mathcal{K}_0(\phi,\eta)$ for $\phi,\eta\in \mathring H(\M)$ and $C,C'\in\R$. 

\subsubsection*{Square-root kernel asymptotics}
In the following we consider the pseudo-differential operator $\sqrt{\mathsf p}$ on an admissible manifold $M$. We recall that since $\mathsf p$ is positive and self-adjoint it has a positive self-adjoint square root $\sqrt{\mathsf p}$ with inverse $\sqrt{\mathsf k_0}\colon \mathring H\to \mathring H^{n/2}$.
By spectral decomposition, $\sqrt{\mathsf k_0}$ is an integral operator with symmetric kernel $g_0(x,y)$, so that \(g_0\) satisfies
\begin{align*}
	\sqrt{\mathsf k_0}\, \phi(x)=\int g_0(x,y)\phi(y)\, d\vol(y).
\end{align*}
By definition of the square root and the integral operators $\mathsf k_0,\sqrt{\mathsf k_0}$, the relation between $k_0$ and $g_0$ is expressed in the following way
\begin{align}\label{eq: green}
	k_0(x,y)=\int g_0(x,z)g_0(y,z)\, d\vol(z).
\end{align}

The next lemma provides the asymptotic behavior of the kernel $g_0(x,y)$ close to the diagonal.

\begin{lemma}\label{est-sqrt-k} $\exists C_0,C_1, \forall x,y$:
	\begin{align}\label{eq: green estimate}
		-C_0\le g_0(x,y)\le C_0+C_1 \frac1{d(x,y)^{n/2}}.
	\end{align}
\end{lemma}
With $C_0$ as above, 
henceforth we will denote
$$\sigma:=\sqrt{|M|} \, C_0.$$

\begin{proof} 
	Let $\tilde g_0$ denote the kernel for $(-\Delta)^{-n/4}$. By Lemma 2.10 in \cite{Del24}
	$$\tilde g_0(x,y)=\frac1{\Gamma(\frac{n}4)}\int_0^\infty t^{\frac n4-1}p_t^0(x,y)\, dt,$$
	where $p_t^0(x,y)=p_t(x,y)-\frac1{|M|}$ denotes the grounded heat kernel.  Then there exist $ C_0,C_1>0$ such that $\forall x,y$
	$$-C_0\le \tilde g_0(x,y)\le C_0+C_1 \frac1{d(x,y)^{n/2}}.$$
	Indeed, for the massive kernel, i.e. $(\alpha-\Delta)^{-n/4}(x,y)$ with $\alpha>0$ the claim was derived in 
	\cite[Lemma 2.10 (iv)]{Del24}. The arguments in the proof of Proposition 2.13 in \cite{Del24} allow to extend these estimates to all $\alpha>-\lambda$, in particular to $\alpha=0$.
	
	Recall that similarly to \eqref{eq: green}, 
	$\tilde k(x,y)=\int \tilde g_0(x,z)\tilde g_0(z,y)d\vol(z)$ is the kernel for the operator $(-\Delta)^{n/2}$. Furthermore, according to \cite{Del24},
	$$|k_0(x,y)-\tilde k(x,y)|\le C, \qquad -C\le \tilde k(x,y)\le C+C' \log \frac1{d(x,y)}.$$
	
	Put $L:=\sqrt{\sf p}$ and $L_0=(-\Delta)^{n/4}$. We know that $L^2-L^2_0$ is a form small perturbation of $L_0^2$. Thus also $L-L_0$ is a form small perturbation of $L_0$, that is, for every $\epsilon>0$ there exists a $C$ such that
	$$-\epsilon L_0-C\le L-L_0\le \epsilon L_0+C$$
	(in the sense of quadratic forms). 
	Indeed,
	\begin{eqnarray*}
		L-L_0&=\frac1{L+L_0}(L^2-L_0^2)\le \frac1{L+L_0}(\epsilon (L^2+L_0^2)+C)\le \epsilon(L+L_0)+C'.
	\end{eqnarray*}
	
	Therefore, since $(1+\epsilon) L-L_0+C\ge0$
	\begin{align*}
		g_0(x,y)&-(1+\epsilon)\tilde g_0(x,y)+C\int \tilde g_0(x,z)g_0(z,y) d\vol(z)\\
		&=\int \tilde g_0(x,z)\Big[(1+\epsilon) L-L_0+C\Big]_zg_0(z,y)d\vol(z)\\
		&\le
		\bigg(\int \tilde g_0(x,z)\Big[(1+\epsilon) L-L_0+C\Big]_z\tilde g_0(z,y)d\vol(z)\bigg)^{1/2}\\
		&\quad\cdot\bigg(\int g_0(x,z)\Big[(1+\epsilon) L-L_0+C\Big]_zg_0(z,y)d\vol(z)\bigg)^{1/2}.
	\end{align*}
	Computing further
	\begin{align*}
		g_0(x,y)&-(1+\epsilon)\tilde g_0(x,y)+C\int \tilde g_0(x,z)g_0(z,y) d\vol(z)\\
		&\le
		\bigg(\int \tilde g_0(x,z)\Big[\big((1+\epsilon)^2-1) L_0+(2+\epsilon)C\Big]_z\tilde g_0(z,y)d\vol(z)\bigg)^{1/2}\\
		&\quad\cdot\bigg(
		\int g_0(x,z)\Big[\big((1+\epsilon)^2-1) L+(2+\epsilon)C\Big]_zg_0(z,y)d\vol(z)
		\bigg)^{1/2}\\
		&\stackrel{(\ast)}=\Big((2\epsilon+\epsilon^2)\tilde g_0(x,y)+(2+\epsilon)C \tilde k(x,y)\Big)^{1/2}\\
		&\quad\cdot
		\Big((2\epsilon+\epsilon^2)g_0(x,y)+(2+\epsilon)C k_0(x,y)\Big)^{1/2}\\
		&\stackrel{(\ast\ast)}\le (\epsilon+\epsilon^2/2)\tilde g_0(x,y)+ (\epsilon+\epsilon^2/2)g_0(x,y)\\
		&\quad+(1+\epsilon/2)C \tilde k(x,y)+(1+\epsilon/2)C k_0(x,y),
	\end{align*}
	where we used the concavity of the square root function $(\ast\ast)$.
	The equality in $(\ast)$ in particular implies that
	$$(2\epsilon+\epsilon^2)g_0(x,y)+(2+\epsilon)C k_0(x,y)\ge0.$$
	Thus the negative part of $g_0$ is dominated by a multiple of the positive part of $k_0$. Hence, 
	\begin{align*}-\int \tilde g_0(x,z)g_0(z,y)d\vol(z)&\le \int \tilde g_0^-\, g_0^+d\vol(z)+\int g_0^-\, \tilde g_0^+d\vol(z) \\
		&\le C' \int g_0^+d\vol(z)+C'\int \tilde k^+\tilde g_0^+d\vol(z)\\
		&= C' \int g_0^-d\vol(z)+C'\int \tilde k^+\tilde g_0^+d\vol(z)\le C''.
	\end{align*}
	Therefore, for small $\epsilon$,
	\begin{align*}
		g_0(x,y)&\le (1+4\epsilon) \tilde g_0(x,y)+2C (1+2\epsilon)\tilde k(x,y)+C'\\
		&\quad-C\frac1{1-\epsilon-\epsilon^2/2}\int \tilde g_0(x,z)g_0(z,y)d\vol(z)\\
		&\le (1+4\epsilon) \tilde g_0(x,y)+2C (1+2\epsilon)\tilde k(x,y)+C'''.
	\end{align*}
	This proves the upper estimate. Interchanging the roles of $g_0, L, k_0$  and $\tilde g_0,L_0,\tilde k$ yields
	$$\tilde g_0(x,y)\le (1+4\epsilon) g_0(x,y)
	+2C (1+2\epsilon)\tilde k(x,y)+C'''$$ which then proves the lower estimate.
\end{proof}

\section{Liouville equation}\label{sec: liou}

GJMS operators are closely related to Branson’s Q-curvature in even dimension,
another important notion in conformal geometry.
The notion of $Q$-curvature was introduced in \cite{Bra85} and its crucial property is its behavior under conformal transformations. It has been shown in \cite{Bra85} that if $\hat g=e^{2u}g$, then
\begin{align}\label{eq: q}
	e^{nu}Q_{\hat g}=Q_g+\frac1{a_n}\mathsf{p}_gu.
\end{align}
As a direct consequence one obtains that
the total $Q$-curvature 
$$Q(M)=\int_MQ\, d\vol$$
is a conformal invariant.
Explicit formulas for Branson's $Q$-curvature are only known in low dimension or for Einstein manifolds. In our sign convention, e.g. if $n=2$ then $Q=\frac12\mathrm{R}$ and if $n=4$ then $Q=-\frac16\Delta\mathrm R-\frac12|\Ric|^2+\frac16\mathrm R^2$. For the round sphere, the equation becomes $Q=(n-1)!$. See \cite[Section 1.4]{Del24} for a comprehensive exposition on these examples.

As for the uniformization theorem for compact surfaces, one can ask for a conformal change of metric $\hat g=e^{2 u}g$ so that the new metric has constant $Q$-curvature $-\Lambda$ for some $\Lambda \in \R$. By the transformation rules \eqref{eq: gjms} and \eqref{eq: q} for $\mathsf p_g$ and $Q_g$ this problem is equivalent to finding a solution of
\begin{align}\label{eq: liouville}
	\frac1{a_n}\mathsf p_gu+Q_g+\Lambda e^{nu}=0, \qquad \Lambda =-\frac{Q(M)}{\int e^{nu}\, d\vol}.
\end{align}

Equation \eqref{eq: liouville} is a natural generalization of the classical two dimensional Liouville equation, which first appeared through the Moser-Trudinger inequality in \cite{Moser}. This equation has been studied under various assumptions in the two-dimensional case in \cite{Moser,kazdan1974curvature}, in the four-dimensional case in  \cite{ChangYang95,Djadli} and in the general even-dimensional case in \cite{Brendle03,Ndiaye}.
We summarize the main results from \cite{Ndiaye} in the following theorem. For this we note that the total $Q$-curvature on the standard $n$-sphere $S^n$ is given by 
$$Q(S^n)=(n-1)!\omega_n,$$
where $\omega_n=\frac{2\pi^{(n+1)/2}}{\Gamma((n+1)/2)}$ is the volume of the $n$-sphere.
\begin{theorem}\label{thm: liouville}
	Let $(M,g)$ be a closed Riemannian manifold of even dimension $n$. Suppose that $\mathsf p$ 
	only annihilates constants 
	and that the total $Q$-curvature satisfies
	\begin{align*}
		Q(M)\neq k(n-1)!\omega_n,
	\end{align*}
	for $k=1,2,\ldots$. Then there exists a smooth function $u\in C^\infty(M)$ such that the $Q$-curvature of the conformal metric $\hat g=e^{2u}g$ equals some constant $-\Lambda$. Equivalently, the pair $(u,\Lambda)$ satisfies \eqref{eq: liouville}. Moreover, $u$ can be obtained as a critical point of the Euler Lagrange functional 
	\begin{align*}
		\Pi(u)=\int \frac 1{2a_n} u \mathsf pu + Q u\, d\vol -  \frac{Q(M)}{n}\log \int e^{nu}\, d\vol
	\end{align*}
	defined on $H^{n/2}(M)$.
\end{theorem}
\begin{proof}
	This follows from a combination of Theorem 1.1 and Corollary 1.4(b) in \cite{Ndiaye}.
\end{proof}

\begin{remark}\label{rem: liouville}
	In \cite{Brendle03} equation \eqref{eq: liouville} has been solved under the condition that $\mathsf p$ is a non-negative operator and $Q(M)<(n-1)!\omega_n$ using a geometric flow. On the other hand, \cite{Ndiaye} shows that under these conditions a Moser-Trudinger type inequality holds, which implies that the functional $\Pi$ is bounded from below, coercive and lower semicontinuous. Hence solutions can be found as global minima using the direct method of Calculus of Variations. This strategy generalizes the four-dimensional approach in \cite{ChangYang95} towards higher even dimensions.  
\end{remark}
It turns out that in our setting solutions $u$ to \eqref{eq: liouville} are unique up to constants.
\begin{corollary}\label{cor: liouville}
	Let $(M,g)$ be an admissible manifold with $Q(M)\leq0$. Then  for every $\kappa\in\R$ there exists a unique
	smooth function $u=u_\kappa\in C^\infty(M)$  with $\langle u\rangle =\kappa$ such that the $Q$-curvature of the conformal metric $\hat g=e^{2u}g$ is constant.
	Indeed, the $Q$-curvature will then be $-\Lambda_\kappa:=\frac{Q(M)}{\int e^{nu_\kappa}\, d\vol}$.
\end{corollary}
Thus in particular, $u_\kappa=u_0+\kappa$ and
$\Lambda_\kappa=e^{-n\kappa}\Lambda_0$.

\begin{proof}
	According to Theorem \ref{thm: liouville} and Remark \ref{rem: liouville} we know that smooth solutions to \eqref{eq: liouville} exist, which are global minimizers of $\Pi$. Furthermore, if $u$ is a global minimizer of $\Pi$, then so is $u+c$ for every constant $c\in\R$.
	
	In order to show uniqueness let $u,v$ be two solutions. Then 
	\begin{align*}
		\Pi\left(\frac{u+v}2\right)=&\frac1{4a_n} \mathfrak p(u)+\frac1{4a_n}\mathfrak p(v)-\frac1{2a_n}\mathfrak p\left(\frac{u-v}2\right)\\
		&+\frac12\int Q u\, d\vol+\frac12\int Q v\, d\vol\\
		&-\frac{Q(M)}n\log\int e^{nu/2}e^{nv/2}\, d\vol\\
		\leq & \frac1{4a_n} \mathfrak p(u)+\frac1{4a_n}\mathfrak p(v)-\frac1{2a_n}\mathfrak p\left(\frac{u-v}2\right)\\
		&+\frac12\int Q u\, d\vol+\frac12\int Q v\, d\vol\\
		&-\frac{Q(M)}{2n}\log\int e^{nu}\, d\vol -\frac{Q(M)}{2n}\log \int e^{nv}\, d\vol\\
		=&\frac12\Pi(u)+\frac12\Pi(v)-\frac1{2a_n}\mathfrak p\left(\frac{u-v}2\right),
	\end{align*}
	where we used H\"older's inequality and the assumption that $Q(M)\leq 0$. Then, since $u$ and $v$ are minimizers we have that $\mathfrak p(\frac{u-v}2)=0$ and by admissibility
	$u=v+c$ for some constant $c$. Prescribing $\langle u\rangle=\kappa$ guarantees uniqueness. For $\Lambda=\Lambda_u$ we obtain
	\begin{align*}
		\Lambda=-Q(M)\Big/\int e^{nu}\, d\vol.
	\end{align*}
\end{proof}

In the case $Q(M)<0$, this existence and uniqueness result can be equivalently re-phrased as follows: for every $\Lambda>0$ there exists  a unique 
smooth function $u=u^{(\Lambda)}\in C^\infty(M)$ such that 
\begin{align*}
	e^{2u}g \text{ has constant $Q$-curvature  }
	-\Lambda.
\end{align*}
Indeed, choosing $\kappa\in\R$ such that $\Lambda=\Lambda_\kappa$ we have $u^{(\Lambda)}=u_\kappa$.

We conclude this section by collecting a few examples of admissible manifolds with non-positive total $Q$-curvature. 

\begin{example}\cite[Example 6.16--6.18]{Del24}
	\begin{enumerate}[$(i)$]
		\item ($n=2$). Every compact Riemannian surface of genus $\geq 1$ satisfies both of these conditions.
		\item ($n=2,6,10,\ldots$). Every compact hyperbolic Riemmanian manifold of dimension $n=4\ell-2$ for some ${\ell}\in \N$ and with $\lambda>\frac{n(n-2)}4$ satisfies both of these conditions.
		\item ($n=4$). Let $M=M_1\times M_2$, where $M_1$ and $M_2$ are compact Riemannian surfaces of constant curvature $k_1$ and $k_2$, respectively. Then
		\begin{enumerate}[$(i)$]
			\item $Q<0$ if and only if 
			\begin{align*}
				|k_1+k_2|<\sqrt{3} |k_1-k_2|,
			\end{align*}
			\item $\mathsf p>0$ on $\mathring{H}$ if $k_1+k_2\geq0$. 
		\end{enumerate}
	\end{enumerate}
\end{example}

\section{The augmented field $h$}\label{sec: aug}
In the following we introduce the \emph{augmented field} $h$. For this recall that $ g_0(x,y)$ defines the kernel of the integral operator $\sqrt{\sf k_0}$. According to Lemma  \ref{est-sqrt-k}, we can choose a constant $\sigma\ge0$ such that for all $x,y\in M$
\begin{align*}
	{g}(x,y):=g_0(x,y)+\frac1{\sqrt{|M|}}\sigma\geq0.
\end{align*}
Formally, the field $h$ is defined by 
$$h(x)= \int_{M}{ g}(x,y)d W(y),$$
where $W$ is white noise,  i.e.\ the $ H^{-n/2-\varepsilon}$-valued centered Gaussian random field with covariance
\begin{equation*}
	\E \big[\langle{W}|{\phi}\rangle \langle{W}|{\eta}\rangle\big]=\langle{\phi}|{\eta}\rangle_{ L^2}   \quad\qquad \forall \phi,\eta\in  L^{2} \fstop
\end{equation*}

The covariance function of $h$ is then formally given by 
$$k(x,y):=\E[h(x)h(y)]=k_0(x,y)+\sigma^2$$
and notably
\begin{align}\label{eq: con}
	k(x,y)=\int g(x,z)g(z,y)\, d\vol(y)\geq0.
\end{align}
More precisely, we define the following.
\begin{definition} The \emph{augmented (co-polyharmonic) field} $h$ is a centered Gaussian random variable taking values in $ H^{-s}$ for some $s>0$ with covariance
	\begin{equation*}
		\E  \Big[\langle h|\phi\rangle\cdot \langle h|\eta\rangle\Big]=\mathcal K(\phi,\eta)
		\quad\qquad \forall u,v\in  H^{s},
	\end{equation*}
	where
	\begin{equation*}
		\mathcal K(\phi,\eta)=\iint k(x,y)\, \phi(x)\, \eta(y)\, d\vol(x)\, d\vol(y).
	\end{equation*}
\end{definition}

For every admissible manifold, such a field exists and is unique in law by the same reasoning as for the co-polyharmonic field $h_0$ introduced in \cite{Del24}, see \cite[Theorem 3.2]{Del24}.
The covariance kernel of the centered Gaussian field $h_0$ is given by $k_0$, i.e. formally
\[
k_0(x,y)={\E}\big[h_0(x)\,h_0(y)\big].
\]
The relationship between $h$ and $h_0$ is given in the following lemma.

\begin{lemma}\label{lem: h and h*}
	The fields \(h_0\) and \(h\) are related by
	\[
	\scalar{h}{\phi}
	\overset{d}{=}
	\scalar{h_0}{\phi}
	+
	\sigma\scalar{\phi}{1} \xi,
	\]
	where $\xi$ is a standard normal random variable independent of $h_0$.	
\end{lemma}

\begin{proof}
	By definition,
	\[
	\E[\scalar{h}{\phi}^2]=\iint \phi(x)\phi(y)k_0(x,y)\, d\vol(x)\, d\vol(y)+\sigma^2\left(\int\phi(x)\, d\vol(x)\right)^2.
	\]
	which is the same as the variance of $\scalar{h}{\phi}
	+
	\sigma \scalar{\phi}{1} \xi$.
\end{proof}	

If $\mathring \P$ and $\P$ denote the laws of $h_0$ and $h$, resp., then $\P$ is the image of 
$\mathring\P\otimes \mathcal N(0,1)$ under the map $(h_0,\xi) \mapsto h_0+\sigma\xi$, where $\xi$ is identified with the constant function.

In particular we have by Proposition 3.9 in \cite{Del24} that the series
\begin{align*}
	h=\sum_{j=1}^\infty \frac{\psi_j\xi_j}{\sqrt{\nu_j}}+\sigma\xi,
\end{align*}
where $(\xi_j)_{j=1}^\infty$ and $\xi$ are independent standard normal random variables,
exists $\P$-almost-surely in $H^{-\eps}$.

Further, we have according to \cite[Proposition 3.7]{Del24} a change of variable formula of Girsanov type. For this let us denote 
\begin{align*}
	&\tilde {\mathsf p}\varphi=\mathsf p\varphi+\frac1{\sigma^2|M|}\langle \varphi\rangle, \\
	&\tilde{\mathfrak p}(\varphi,\varphi)=\int \varphi\, \tilde{\mathsf p}\varphi\, d\vol=\mathfrak p(\varphi,\varphi)+\frac1{\sigma^2}\langle \varphi\rangle^2.
\end{align*}
Note that if $\varphi\in\mathring H^{n/2}$, $\tilde{\mathsf p}=\mathsf p$ and $\tilde{\mathfrak p}(\varphi,\varphi)={\mathfrak p}(\varphi,\varphi)$.
\begin{lemma}
	If $\varphi\in H^{n/2}$ and $h\sim \mathbb P$, 
	then $h+\varphi$ is distributed according to 
	\begin{equation}\label{eq: girsanov0}
		\exp\paren{\scalar{h}{\tilde {\mathsf p}\varphi}-\frac12 \tilde{\mathfrak p}(\varphi,\varphi) } \, d\mathbb P(h).
	\end{equation}
\end{lemma}

\begin{proof}
	
	Let $f$ be a bounded measurable test function on $H^{-s}$. By Lemma \ref{lem: h and h*}, we may realize 
	\begin{align*}
		h=T(h_0,\xi)=h_0+\sigma\xi,
	\end{align*}
	where $h_0\sim \mathring \P$ and $\xi\sim \mathcal N(0,1)$ is independent of $h_0$.
	By the  Cameron-Martin theorem \cite[Proposition 3.7]{Del24} we have that $h_0+\mathring\varphi$ is distributed according to
	\begin{align*}
		\exp(\scalar{h_0}{\mathsf p\mathring\varphi}-\frac12\mathfrak p(\mathring\varphi,\mathring\varphi))\, d\mathring\P(h_0)
	\end{align*}
	for every $\mathring\varphi\in\mathring H^{n/2}$.
	Then by definition of the pushforward
	\begin{align*}
		&\int f\, d(h+\varphi)_\#\P=\int f(h+\varphi)\, d\P\\
		=&\int f(T(h_0,\xi)+\pi(\varphi)+\langle\varphi\rangle)\, d\mathring\P(h_0)\, d\mathcal N(\xi)\\
		=&\int f(T(h_0,\xi) + \langle \varphi\rangle)\exp(\scalar{h_0}{\mathsf p\varphi}-\frac12\mathfrak p(\varphi,\varphi))\, d\mathring\P(h_0)\, d\mathcal N(\xi)\\
		=&\int f(T(h_0,\xi))\exp(\scalar{h_0}{\mathsf p\varphi}-\frac12\mathfrak p(\varphi,\varphi))\exp\left(\frac{\xi\langle\varphi\rangle}{\sigma}-\frac{\langle\varphi\rangle^2}{2\sigma^2}\right)\, d\mathring\P(h_0)\, d\mathcal N(\xi)\\
		=&\int f(h)\exp\paren{\scalar{h}{\mathsf p \varphi+\frac1{\sigma^2|M|}\langle\varphi\rangle}-\frac12 (\mathfrak p(\varphi,\varphi)+\frac{\langle\varphi\rangle^2}{\sigma^2})}\, d\P(h),
	\end{align*}
	where we used in the last line that 
	$\mathsf p$ is a self-adjoint operator, which annihilates constants.
	
\end{proof}

In the following we introduce another approximation $h_{\ell}$ of the field $h$, the so-called \emph{white noise truncation}.

\subsection{White noise truncation}\label{sec: wn trunc}

\begin{definition}
	The \emph{white noise truncation} $h_{\ell}$ of the augmented field $h$ is defined by the centered Gaussian field
	\begin{align}\label{eq: white noise trunc}
		h_{\ell}(x) := & \int_{M\setminus B_{1/\ell}(x)}{ g}(x,y)d W(y), \quad x \in M.
	\end{align}	
	The correlation function is given by 
	\begin{align}\label{eq: kernel wn decomp}
		k_{\ell}(x,y):=\E[h_{\ell}(x)h_{\ell}(y)]=\int_{M\setminus(B_{1/\ell}(x)\cup B_{1/\ell}(y))}{ g}(x,z) { g}(y,z)\,d\vol(z)\ge0.
	\end{align}
\end{definition}

\begin{lemma}
	Let  $(h_{\ell})_{\ell\in\N}$ be
	the white noise truncation  given in \eqref{eq: white noise trunc}. The associated covariance kernels $(k_{\ell})_{\ell\in\N}$ satisfy
	\begin{enumerate}[$(i)$]
		\item as ${\ell}\to\infty$, 
		\begin{align}\label{eq: convergence}
			k_{\ell}\to k \quad\text{ in }L^1(M^2,\vol\otimes \vol),
		\end{align}
		\item there exist $C\geq0$ such that for all ${\ell}\in \N$ and all $x,y\in M$,
		\begin{align}\label{eq: domination}
			k_{\ell}(x,y)\leq C\log\left(\frac1{d(x,y)}\wedge \ell\right)+C,
		\end{align}
		\item for all ${\ell}\in\N$ there exists a constant $C_{\ell}>0$ such that
		\begin{align}\label{eq: k cont}
			|k_{\ell}(x,x)+k_{\ell}(y,y)-2k_{\ell}(x,y)|\leq C_{\ell} d(x,y), \quad x,y\in M.
		\end{align}
	\end{enumerate}
\end{lemma}

\begin{proof}
	$(i)$ Since 
	\begin{align*}
		k_\ell(x,y)\nearrow
		\int_{M}{ g}(x,z) { g}(y,z)\,d\vol(z)=k(x,y),
	\end{align*}
	by \eqref{eq: con}, the claim follows by monotone convergence.
	
	$(ii)$ By \eqref{eq: log} there exists a $C>0$ such that for all $x,y\in M$
	\begin{align*}
		k_{\ell}(x,y)\leq \log\frac1{d(x,y)} +C.
	\end{align*}
	Moreover,  there exist constants 
	$C,C'$
	such that $\forall x,y,\ell$
	\begin{align*}
		k_\ell(x,y)\le C+C'\log\ell.
	\end{align*}
	Indeed, by \eqref{eq: green estimate},
	\begin{align*}
		k_\ell(x,y)
		&\le\bigg(\int_{M\setminus B_{1/\ell}(x)} { g}^2(x,z)\,d\vol(z) \bigg)^{1/2}\cdot
		\bigg(\int_{M\setminus B_{1/\ell}(y)} { g}^2(y,z)\,d\vol(z) \bigg)^{1/2}\\
		&\le\bigg(\int_{M\setminus B_{1/\ell}(x)} \Big(C_0+C_1 \frac1{d(x,z)^{n/2}} \Big)^2\,d\vol(z)\bigg)^{1/2}\\
		&\qquad\cdot
		\bigg(\int_{M\setminus B_{1/\ell}(y)} \Big(C_0+C_1 \frac1{d(y,z)^{n/2}}\Big)^2\,d\vol(z) \bigg)^{1/2}\\
		&\le C+C'\int_{1/\ell}^\infty \frac1{r^n}r^{n-1}dr=C+C'\log\ell.
	\end{align*}
	Thus, \eqref{eq: domination} follows.
	
	$(iii)$
	We decompose $k_{\ell}(x,x)+k_{\ell}(y,y)-2k_{\ell}(x,y)$ using \eqref{eq: kernel wn decomp} and obtain
	\begin{align*}
		&|k_{\ell}(x,x)+k_{\ell}(y,y)-2k_{\ell}(x,y)|\\
		\leq&\int_{B_{1/\ell}(y)\setminus B_{1/\ell}(x)}g^2(x,z)\, d\vol(z) +\int _{B_{1/\ell}(x)\setminus B_{1/\ell}(y)}g^2(y,z)\, d\vol(z)\\
		&+\int_{M\setminus B_{1/\ell}(x)\setminus B_{1/\ell}(y)}|g(x,z)-g(y,z)|^2\, d\vol(z).
	\end{align*}
	From now on we write $\eps=1/\ell$ and we treat the first two terms and the last term separately.
	\begin{claim}\label{claim: first bound}
		\begin{align*}
			\int_{B_\eps(y)\setminus B_\eps(x)}g^2(x,z)\, d\vol(z) +\int _{B_\eps(x)\setminus B_\eps(y)}g^2(y,z)\, d\vol(z)\leq C \frac{d(x,y)}\eps.
		\end{align*}
	\end{claim}
	\begin{proof}[Proof of Claim \ref{claim: first bound}]
		If $d(x,y)>2\eps$ we get by \eqref{eq: green estimate}
		\begin{align*}
			&\int_{B_\eps(y)\setminus B_\eps(x)}g^2(x,z)\, d\vol(z) +\int _{B_\eps(x)\setminus B_\eps(y)}g^2(y,z)\, d\vol(z)\\
			\leq &\int_{B_\eps(y)}g^2(x,z)\, d\vol(z) +\int _{B_\eps(x)}g^2(y,z)\, d\vol(z)\\
			\leq &C\eps^nd(x,y)^{-n}\leq C\frac{d(x,y)}\eps.
		\end{align*}
		If $d(x,y)\leq 2\eps$, then again by \eqref{eq: green estimate}
		\begin{align*}
			&\int_{B_\eps(y)\setminus B_\eps(x)}g^2(x,z)\, d\vol(z) +\int _{B_\eps(x)\setminus B_\eps(y)}g^2(y,z)\, d\vol(z)\\
			\leq &C\int_{B_\eps(y)\setminus B_\eps(x)}d(x,z)^{-n}\, d\vol(z) +C\int _{B_\eps(x)\setminus B_\eps(y)}d(y,z)^{-n}\, d\vol(z)\\
			\leq &C\int_{B_\eps(y)\setminus B_\eps(x)}\eps^{-n}\, d\vol(z) +C\int _{B_\eps(x)\setminus B_\eps(y)}\eps^{-n}\, d\vol(z)\\
			\leq &C\eps^{-n}\eps^{n-1}d(x,y)
			\leq C\frac{d(x,y)}\eps,
		\end{align*}
		which proves the claim.
	\end{proof}

	\begin{claim}\label{claim: second bound}
		\begin{align*}
			\int_{M\setminus B_\eps(x)\setminus B_\eps(y)}|g(x,z)-g(y,z)|^2\, d\vol(z)\leq C\frac{d^2(x,y)}{\eps^2}.
		\end{align*}
	\end{claim}
	\begin{proof}[Proof of Claim \ref{claim: second bound}]
		In order to show the estimate, we can find a path $p$ from $x$ to $y$ and a constant $C=C(M,g)>0$ such that the length $L(p)$ is bounded by $Cd(x,y)$ and for all $z$
		\begin{align*}
			\mathrm{dist}(p,z)\geq \frac1C(d(x,z)\wedge d(y,z)).
		\end{align*}
		This follows from standard bounded-geometry arguments on compact manifolds: Each minimizing geodesic can be modified to avoid $B(z,r)$, where $r\sim d(x,z)\wedge d(y,z)$, with uniformly controlled length, see e.g. \cite{petersen2006riemannian} and \cite{burago2001course}.
		
		Then, applying the mean value theorem on the path $p$, there exists a $\zeta$ depending on $x,y,z$ such that
		\begin{align*}
			&\int_{M\setminus B_\eps(x)\setminus B_\eps(y)}|g(x,z)-g(y,z)|^2\, d\vol(z)\\
			&\leq C\int_{M\setminus B_\eps(x)\setminus B_\eps(y)}|Dg(\zeta,z)|^2 d(x,y)^2\, d\vol(z).
		\end{align*}
		Note that due to \eqref{eq: green estimate}, we have 
		\begin{align*}
			|D g|(x,y)\leq Cd(x,y)^{-n/2-1}(x,y).
		\end{align*}
		Then, using that $d(\zeta,z)\geq d(x,y)\wedge d(y,z)$
		\begin{align*}
			&\int_{M\setminus B_\eps(x)\setminus B_\eps(y)}|g(x,z)-g(y,z)|^2\, d\vol(z)\\
			&\leq C\int_{M\setminus B_\eps(x)\setminus B_\eps(y)}(d(x,z)^{-n-2}+d(y,z)^{-n-2})d(x,y)^2\, d\vol(z)\\
			&\leq  C\int_{M\setminus B_\eps(x)}d(x,z)^{-n-2}d(x,y)^2\, d\vol(z)\\
			&\qquad+ C\int_{M\setminus B_\eps(y)}d(y,z)^{-n-2}d(x,y)^2\, d\vol(z)\\
			&\leq \frac{d(x,y)^2}{\eps^2},
		\end{align*}
		which proves the claim.
	\end{proof}
	
	Combining Claim \ref{claim: first bound} and \ref{claim: second bound}, we have shown since $M$ is compact
	\begin{align*}
		&|k_{\ell}(x,x)+k_{\ell}(y,y)-2k_{\ell}(x,y)|
		\leq C \frac{d(x,y)}\eps+C\frac{d^2(x,y)}{\eps^2}\leq C\frac{d(x,y)}{\eps^2}.
	\end{align*}

\end{proof}

\begin{proposition}\label{prop: wn decomp}
	Let $(h_{\ell})_{\ell\in\N}$ the white noise truncation given in \eqref{eq: white noise trunc}.
	\begin{enumerate}[$(i)$]
		\item
		For each ${\ell}\in\N$, \(h_{\ell}\) is a centered continuous Gaussian field on $M$.
		\item 
		For all $\ell,m\in\N$ with $m>\ell$ and for all $x\in M$, $(h_m-h_{\ell})(x)$ is independent from $h_{\ell}(x)$. 
		\item For all $\ell,m\in\N$ with $m>\ell$, the correlation function of $h_m-h_{\ell}$ is non-negative, i.e. for all $x,y\in M$
		\begin{equation}\begin{aligned}\label{eq: pos corr}
				\E\Big[(h_m-h_\ell)(x)\, (h_m-h_\ell)(y)\Big] \quad\ge0.
		\end{aligned}\end{equation}
		\item
		As \(\ell\to \infty\), the fields \(h_{\ell}\) converge to the field \(h\) on \(M\) in the sense that
		\begin{align*}
			\scalar{h_{\ell}}{\varphi}\to \scalar{h}{\varphi} \text{ in }L^2(\Omega) \quad \forall \varphi\in L^2(M).
		\end{align*}
		\item  for every bounded nonnegative test function $\varphi\colon M\to\R$,
		\begin{align}\label{eq: variance}
			\mathbb{E}\!\left[\scalar{ h-h_{\ell}}{\varphi}^2\right]
			= O(\ell^{-1}),
		\end{align}
	\item For $\ell\in\N$ and $\gamma^2<\frac n8$
	\begin{align}\label{eq: white-noise}
		\int \Big[e^{\gamma^2\bar G_\ell(x,y)}-1\Big]\, d\vol(y)\leq C\gamma^2\ell^{-1}.
	\end{align}
	where $\bar G_\ell$ is the correlation function of $h-h_\ell$.
	\end{enumerate}
\end{proposition}

\begin{proof}
	$(i)$ In order to check continuity,
	note that 
	\begin{align*}
		\E[|h_{\ell}(x)-h_{\ell}(y)|^2]=k_{\ell}(x,x)+k_{\ell}(y,y)-2k_{\ell}(x,y).
	\end{align*}
	In particular by \eqref{eq: k cont}, $h_{\ell}(x)-h_{\ell}(y)$ is a centered Gaussian random variable with covariance dominated by $C_{\ell}d(x,y)$. Therefore, it has finite moments of all orders $p>1$, and there exists a constant $C_{l,p}>0$ such that 
	\begin{align*}
		\E[|h_{\ell}(x)-h_{\ell}(y)|^p]\leq C_{l,p}d(x,y)^p, \quad x,y\in M.
	\end{align*}
	By the Kolmogorov-Chentsov Theorem, $h_{\ell}$ has a continuous version almost surely.
	
	$(ii)$ Note that
	\begin{align}\label{eq: increment}
		(h_m-h_{\ell})(x)=\int_{B_{1/\ell}(x)\setminus B_{1/m}(x)}{ g}(x,y)\, dW(y) 
	\end{align} 
	is a centered Gaussian random variable. The independence of $h_{\ell}(x)$ follows from \eqref{eq: white noise trunc} and
	\begin{align*}
		\E[(h_m-h_{\ell})(x)h_{\ell}(x)]=0.
	\end{align*}
	
	$(iii)$  By \eqref{eq: increment} we find
	\begin{align*}
		\E\Big[(h_m-h_\ell)(x)\, &(h_m-h_\ell)(y)\Big]\\
		&=\int_{(B_{1/\ell}(x)\setminus B_{1/m}(x))\cap (B_{1/\ell}(y)\setminus 
			B_{1/m}(y))}{ g}(x,z) { g}(y,z)\,d\vol(z) \ge0.
	\end{align*}
	
	$(iv)$ By definition of \(h_{\ell}\) and $k_{\ell}$ we have 
	\[
	\E[\scalar{ h_{\ell}}{\varphi}^2] =\iint k_{\ell}(x,y)\varphi(x)\varphi(y)\, d\vol(x)\, d\vol(y)
	\]
	and as $\ell\to\infty$
	$$k_\ell(x,y)\nearrow
	\int_{M}{ g}(x,z) { g}(y,z)\,d\vol(z)=k(x,y)$$ 
	by \eqref{eq: green}. Thus,
	\[
	\scalar{h_{\ell}}{\varphi} \xrightarrow{L^2(\Omega)} \scalar{h}{\varphi}.
	\]
	$(v)$
	Note that 
	$$\scalar{ h-h_{\ell}}{\varphi}=\Big\langle W\Big|
	\int_M { g}(x,.)1_{B_{1/\ell}(x)}(.)\varphi(x)d\vol(x)\Big\rangle$$
	and thus
	\begin{align*}
		\E[\scalar{ h-h_{\ell}}{\varphi}^2]&=
		\Big\|\int_M { g}(x,.)1_{B_{1/\ell}(x)}(.)\varphi(x)d\vol(x)\Big\|_{L^2}^2\\
		&\le \|\varphi\|^2\cdot \int\int_{B_{1/\ell}(x)}{ g}^2(x,y)d\vol(y)d\vol(x)\\
		&\le \|\varphi\|^2\cdot \int_{B_{1/\ell}(0)\subset M^2}\Big[C_0+C_1\frac1{|x-y|^{n/2}}\Big]^2d\vol(y)d\vol(x)\\
		&\le \|\varphi\|^2\cdot \int_0^{1/\ell}\Big[C'_0+C'_1\frac1{r^n}\Big]r^{2n-1}dr\\
		&\le \|\varphi\|^2\cdot C\,\Big(\frac1\ell\Big)^n.
	\end{align*}

$(vi)$ According to \eqref{eq: pos corr}, the correlation functions of $h_m-h_\ell$ are nonnegative for all $m\ge\ell$. Passing to the limit $m\to\infty$ yields
$\bar G_\ell(x,y)\ge0$ for all $x,y$.

Then for $\gamma^2<\frac n8$,
\begin{align*}
	&	\int \Big[e^{\gamma^2\bar G_\ell(x,y)}-1\Big]\, d\vol(y)\le\gamma^2\, \int \Big[\bar G_\ell(x,y)+\gamma^2 \bar G_\ell(x,y)^2e^{\gamma^2\bar G_\ell(x,y)}\Big]d\vol(y)\\
	&\le\gamma^2\, \int \Big[\bar G_\ell(x,y)+\gamma^{1/2} \bar G_\ell(x,y)^{1/2} e^{2\gamma^2\bar G_\ell(x,y)}\Big]d\vol(y)\\
	&\le \gamma^2\, \int  \bar G_\ell(x,y)d\vol(y)+ \gamma^{5/2}\, \Big(\int  \bar G_\ell(x,y)d\vol(y)\Big)^{1/2}\cdot \Big(\int e^{\frac n2\bar G_\ell(x,y)}d\vol(y)\Big)^{1/2}
\end{align*}
with
$$\int e^{\frac n2\bar G_\ell(x,y)}d\vol(y)\le \int e^{\frac n2 k(x,y)}d\vol(y)\le\int \frac1{d(x,y)^{n/2}}d\vol(y)\le C$$
and
\begin{align*}\int \bar G_\ell(x,y)d\vol(y)&=\int_M \int_{B_{1/\ell}(x,y)}{ g}(x,z){ g}(z,y)d\vol(z)d\vol(y)\\
	&\le\int_M \int_{B_{1/\ell}(x)}{ g}(x,z){ g}(z,y)d\vol(z)d\vol(y)\\
	&\le C \int_{B_{1/\ell}(x)}{ g}(x,z)d\vol(z)\\
	&\le C' \int_{1/\ell}^\infty r^{-n/2} r^{n-1}dr=C'' \ell^{-n/2} 
\end{align*}
\end{proof}

\subsection{Liouville measure}
Next we define the \emph{Liouville measure} associated to the augmented field $h$, which is based on the so-called \emph{Gaussian multiplicative chaos} formally given by
\begin{align*}
	e^{\gamma h(x)-\frac{\gamma^2}2k(x,x)}\, d\vol(x)
\end{align*}
for some suitable $\gamma\in \R$. 
The Gaussian multiplicative chaos (GMC) goes back to Kahane in \cite{Kah85} and is investigated as the Liouville measure in Liouville quantum gravity \cite{Ber17,DKRV16, DS11, LevyOz, RhoVar14}.

Following the approach by Shamov \cite{Sha16}, who gave a first canonical definition of GMC that does not depend on any specific approximation scheme by employing a simple transformation rule,
the authors in  \cite{Del24} introduce the Liouville measure for the co--polyharmonic field $h_0$.

Similarly to \cite[Theorem 4.1]{Del24}, we define for the augmented field:

\begin{definition}
	For $0<\gamma<\sqrt{2n}$ and the Gaussian field $h$, the Liouville Quantum measure $\mu^{\gamma h}$ is a measurable map
	\begin{equation*}
		\mu^{\gamma\bullet} \colon H^{-\varepsilon} \to \mathcal{M}_{b}(M), \qquad f \mapsto \mu^{\gamma f}\comma
	\end{equation*}
	satisfying
	
	\begin{enumerate}[$(i)$]
		\item
		for $\P$-a.e.~$h$ and every $\varphi \in \mathring H^{n/2}$,
		\begin{equation}\label{e:liouville-cm-shift}
			\mu^{\gamma(h+ \varphi)} = e^{\gamma\varphi}\, \mu^{\gamma h} \fstop
		\end{equation}
		
		\item
		for all Borel measurable $f \colon H^{-\varepsilon} \times M \to [0,\infty]$, we have that
		\begin{equation}\label{e:liouville-campbell}
			\E \int f(h, x) d\mu^{\gamma h}(x) = \E \int f\tparen{h + \gamma k(x,\emparg), x} \, d\vol(x) \fstop
		\end{equation}
	\end{enumerate}
\end{definition}
Moreover, for all $p \in \tparen{-\infty, \frac{2n}{\gamma^{2}}}$,
\begin{equation}\label{eq: liouville-moments}
	\E\tbraket{ {\mu^{\gamma h}(M)}^{p} } < \infty \fstop
\end{equation}
\begin{remark}
	The existence and uniqueness follows with the same arguments as presented in \cite{Del24}. However, let us note that
	$k(x,y)$ is of \emph{$\sigma$-positive type} since
	\begin{align*}
		k(x,y)=\int_{M\setminus (B_1(x)\cup B_1(y))}{ g}(x,y)\, d\vol(y) + \sum_{l=1}^\infty \int _{A_{\ell}(x,y)}{ g}(x,z){ g}(y,z)\, d\vol(z), 
	\end{align*}
	where $A_{\ell}(x,y)=(B_{1/\ell}(x)\cup B_{1/\ell}(y))\setminus(B_{1/(\ell+1)}(x)\cup B_{1/(\ell+1)}(y))$ are disjoint annular regions.
	In particular Kahane's original work \cite{Kah85} on GMC is directly applicable, which means there exists a subcritical GMC and it is unique by \cite{Sha16}.
\end{remark}

Note that by Lemma \ref{lem: h and h*} the Liouville measure associated to the augmented field $h$ differs from the definition $\mu^{\gamma h_0}$ in \cite{Del24} only by a global random multiplicative factor 
\begin{align*}
	\mu^{\gamma h}=e^{\gamma \sigma\xi-\frac{\gamma^2}2\sigma^2}\mu^{\gamma h_0}.
\end{align*}

In the next theorem, we show that the Liouville measure can be approximated by considering absolutely continuous random measures
\begin{align}\label{gmc-approx}
	d\mu^{\gamma h_{\ell}}(x):=\exp(\gamma h_{\ell}(x)-\frac{\gamma^2}2k_{\ell}(x,x))\, d\vol(x),
\end{align}
where $h_{\ell}$ is
the Gaussian field  given in \eqref{eq: white noise trunc} with associated covariance kernel $k_{\ell}$.

\begin{theorem}\label{t:WhiteNoiseApproxGMC}
	Let $0<\gamma<\sqrt{2n}$ and
	consider the white noise truncation $\seq{h_{\ell}}_{\ell \in \mathbb{N}}$ given in \eqref{eq: white noise trunc} with covariance kernel $\seq{k_{\ell}}_{\ell \in \mathbb{N}}$.
	Let $\seq{\mu^{\gamma h_{\ell}}}_{l \in \mathbb{N}}$ be as in \eqref{gmc-approx}.
	Then for all continuous bounded $\varphi\colon M\to\R$ 
	\begin{equation*}
		\mu^{\gamma h_{\ell}}(\varphi)\to \mu^{\gamma h}(\varphi)\qquad\P\text{-a.s. and in }L^1(\P)\text{ as }\ell\to\infty\fstop
	\end{equation*}
	
\end{theorem}

\begin{proof}%
	A standard computation yields that for all $m>\ell$
	\begin{align*}
		\E[\mu^{\gamma h_m}(\varphi)|\sigma(h_\ell)]=\mu^{\gamma h_\ell}(\varphi).
	\end{align*}
	In particular $(\mu^{\gamma h_\ell}(\varphi))_\ell$ is a nonnegative martingale.
	Since $k_l(x,y)\leq k(x,y)$, Kahane's comparison inequality \cite[Theorem 28]{Sha16} implies that for every convex $f\colon \R_+\to\R_+$
	\begin{align*}
		\E[f(\mu^{\gamma h_l}(M))]\leq \E[f(\mu^{\gamma h}(M))].
	\end{align*}
	Choosing $f(t)=t^{1+\eps}$ for some $\eps>0$, we obtain by \eqref{eq: liouville-moments} uniform integrability of $(\mu^{\gamma h_\ell}(M))_\ell$, and therefore of $(\mu^{\gamma h_l}(\varphi))_\ell$. The martingale convergence theorem then yields the assertion.
\end{proof}

\subsection{Polyakov--Liouville measure} 
The \emph{Polyakov--Liouville measure} is a fundamental object in the probabilistic formulation of two-dimensional conformal field theory and Liouville quantum gravity. Originally introduced by Polyakov in bosonic string theory, the rigorous construction of this measure has been carried out recently in various settings, see e.g. \cite{DKRV16,DRV16,GRV19}.
In dimension greater than $2$, analogous constructions were considered on spheres in \cite{Cercle}, previously anticipated in the physics literature \cite{LevyOz}.

On arbitrary
even-dimensional admissible manifolds, the Polyakov--Liouville meaure 
$\boldsymbol{\nu}_\gamma^0$ associated to the co-polyharmonic field $h_0$ has been introduced in \cite{Del24} and is formally given by
\begin{equation*}
	\boldsymbol{\nu}^0_\gamma(dh) = \exp(-S(h)) dh
\end{equation*}
with (non-existing) Lebesgue measure $dh$ on the set of fields 
and the action functional
\begin{equation*}
	S(h) \coloneqq   \int_M\Big( \frac12\,h\mathsf ph +\Theta\,Q h + {m} e^{\gamma h}\Big) d \vol\comma
\end{equation*}
where  $m,\Theta,\gamma$ are suitable parameters. 

Accordingly, we define the Polyakov-Liouville measure $\boldsymbol{\nu}_{\kappa,\gamma}$ associated to the augmented field $h$ and $\kappa\in\R$.

\begin{definition}
	Let $\kappa\in\R$ and $h$ be the augmented field. Let $0<\gamma<\sqrt{2n}$ and $m>0,\Theta\in\R$. 
	The \emph{Polyakov--Liouville measure} $\boldsymbol{\nu}_{\kappa,\gamma}$ is defined as
	\begin{equation*}
		d{\boldsymbol{\nu}}_{\kappa,\gamma}(h)\coloneqq \exp\Big(-\Theta\scalar{h}{Q}- m \,\mu^{\gamma h}(M)\Big)\,d\P_{n\kappa/\gamma}(h)
	\end{equation*}
	where $\mu^{\gamma h}$ denotes the Liouville measure and $\P_{n\kappa/\gamma}$ is the law of $h+\frac n\gamma\kappa$, which is given by 
	$$\P_{n\kappa/\gamma}=(T_{\kappa,\gamma})_\# \P$$
	with $T_{\kappa,\gamma}\colon h \mapsto h+\frac n\gamma\kappa$.
\end{definition}
This is slightly different from the definition of the Polyakov--Liouville measure as introduced in \cite{Del24}
which rephrased in the current notation reads as 
\begin{equation*}
	d{\boldsymbol{\nu}}^0_\gamma(h)\coloneqq \exp\Big(-\Theta\scalar{h_0}{Q}- m \,\mu^{\gamma h_0}(M)\Big)\,d\mathring\P(h_0).
\end{equation*}
Note that with Lemma \ref{lem: h and h*} and $\mathcal N:=\mathcal N(0,1)$,
\begin{equation}\begin{aligned}\label{eq: comp pol}
		\int& f(h)
		d{\boldsymbol{\nu}}_{\kappa,\gamma}(h)
		=
		\int f(h_0+\sigma\xi+\frac n\gamma\kappa)\\
		&\cdot  \exp\Big(-\Theta\scalar{h_0+\sigma\xi+\frac n\gamma\kappa}{Q}- m \,e^{\gamma \xi+n\kappa}\,\mu^{\gamma h_0}(M)\Big)\,d\mathring\P(h_0)\, d\mathcal N(\xi).
\end{aligned}\end{equation}

It turns out that $\boldsymbol{\nu}_{\kappa,\gamma}$ is finite and in particular,
\begin{align*}
	\Q_{\kappa,\gamma}=\frac1{Z_{\kappa,\gamma}}\boldsymbol\nu_{\kappa,\gamma},
\end{align*}
is a probability measure,
where $Z_{\kappa,\gamma}$ is the normalizing constant.
\begin{lemma}\label{lemma: finite}
	The measure $\boldsymbol{\nu}_{\kappa,\gamma}$ is finite.
\end{lemma}
\begin{proof}
	By \eqref{eq: comp pol}, we have for 
	$\mathcal N_{\kappa',\sigma^2}=\mathcal N(\kappa',\sigma^2)$ with $\kappa'=\frac n\gamma \kappa$,
	\begin{align*}
		\int 1\, d\boldsymbol{\nu}_{\kappa,\gamma}
		=	&\int_\R\int 	\exp\Big(-\Theta\scalar{h_0+r}{Q}- m \,e^{\gamma r}\mu^{\gamma h_0}(M)\Big)\,d\mathring\P(h_0)\, d\mathcal N_{\kappa',\sigma^2}(r)\\
		=&\int \exp(-\Theta\scalar{h_0}{Q})\E_{\mathcal N_{\kappa',\sigma^2}}[\exp(-\Theta Q(M)X-m\mu^{\gamma h_0}(M) e^{\gamma X})]\, d\mathring\P(h_0).
	\end{align*}
	For the expectation $\E_{\mathcal N_{\kappa',\sigma^2}}$ we further compute
	\begin{align*}
		&E_{\mathcal N_{\kappa',\sigma^2}}[\exp(-\Theta Q(M)X-m\mu^{\gamma h_0}(M) e^{\gamma X})]\\
		&\leq\E[\exp(-\Theta Q(M)X)]=e^{-\Theta Q(M)\kappa'+\Theta^2Q(M)^2\sigma^2/2}<\infty
	\end{align*}
	where we used that $m,\mu^{\gamma h_0}(M)\geq0$. This finishes the proof, since
	\begin{align*}
		\int \exp(-\Theta\scalar{h_0}{Q})\, d\mathring\P(h_0)=e^{\frac{\Theta^2}2\mathcal K_0(Q,Q)}<\infty.
	\end{align*}
\end{proof}

Recall that $\Q_{\kappa,\gamma}^{\mathrm{sc}}$ denotes the distribution of the rescaled fields $\hbar=\frac\gamma n h$
under the normalized Polyakov-Liouville measure $\Q_{\kappa,\gamma}$ with particular choice of constants
$$\Theta=\frac{na_n}{\gamma},  \quad m=\frac{na_n\Lambda_\kappa}{\gamma^2}.$$
In the next lemma we show that $Q_{\kappa,\gamma}^{\mathrm{sc}}$ can be obtained from  $ Q_{0,\gamma}^{\mathrm{sc}}$ through shifting by $\kappa$.
\begin{lemma}\label{lemma: Qkappa}
	Let $\Theta=\frac{na_n}{\gamma}$ and $m=\frac{na_n\Lambda_\kappa}{\gamma^2}$. We have 
	\begin{align*}
		Z_{\kappa,\gamma}=e^{-\Theta\scalar{\frac n\gamma\kappa}{Q}}Z_{0,\gamma}.
	\end{align*}
	Moreover, if $\hbar\sim Q_{0,\gamma}^{\mathrm{sc}}$, then $\hbar+\kappa\sim Q_{\kappa,\gamma}^{\mathrm{sc}}$ for $\kappa\in \R$.
\end{lemma}
\begin{proof}
	Let
	\[
	c_\kappa:=\frac{n}{\gamma}\kappa.
	\]
	By definition of the shifted Gaussian law,
	\[
	\mathbb P_{n\kappa/\gamma}
	=(\tau_{c_\kappa})_\#\mathbb P,
	\qquad
	\tau_{c_\kappa}(h)=h+c_\kappa,
	\]
	where $\mathbb P$ denotes the law of the augmented field
	$h=h_0+\sigma\xi$.
	
	Recall that the Liouville measure satisfies the multiplicative shift property
	\[
	\mu^{\gamma(h+c)}(M)
	=
	e^{\gamma c}\mu^{\gamma h}(M)
	\]
	for every constant $c$. Hence
	\[
	\mu^{\gamma(h+c_\kappa)}(M)
	=
	e^{\gamma c_\kappa}\mu^{\gamma h}(M)
	=
	e^{n\kappa}\mu^{\gamma h}(M).
	\]
	
	Since
	\[
	m=\frac{na_n\Lambda_\kappa}{\gamma^2},
	\]
	and by the choice $\Lambda_\kappa=e^{-n\kappa}\Lambda_0$, we obtain
	\[
	m\,\mu^{\gamma(h+c_\kappa)}(M)
	=
	\frac{na_n\Lambda_0}{\gamma^2}\mu^{\gamma h}(M).
	\]
	
	Therefore
	\begin{align*}
		Z_{\kappa,\gamma}
		&=
		\int
		\exp\Big(
		-\Theta \langle h,Q\rangle
		-m\,\mu^{\gamma h}(M)
		\Big)\,
		d\mathbb P_{n\kappa/\gamma}(h)
		\\
		&=
		\int
		\exp\Big(
		-\Theta \langle h+c_\kappa,Q\rangle
		-\frac{na_n\Lambda_0}{\gamma^2}\mu^{\gamma h}(M)
		\Big)\,
		d\mathbb P(h)
		\\
		&=
		e^{-\Theta\langle c_\kappa,Q\rangle}
		\int
		\exp\Big(
		-\Theta \langle h,Q\rangle
		-\frac{na_n\Lambda_0}{\gamma^2}\mu^{\gamma h}(M)
		\Big)\,
		d\mathbb P(h)
		\\
		&=
		e^{-\Theta\langle \frac{n}{\gamma}\kappa,Q\rangle}
		\,Z_{0,\gamma}.
	\end{align*}
	
	This proves the first claim.
	
	For the second claim, let $F$ be a bounded measurable functional on the
	space of rescaled fields. Then
	\begin{align*}
		\mathbb E_{Q_{\kappa,\gamma}^{\mathrm{sc}}}[F(\hbar)]
		&=
		\frac1{Z_{\kappa,\gamma}}
		\int
		F\!\left(\frac{\gamma}{n}h\right)
		e^{-\Theta\langle h,Q\rangle-m\mu^{\gamma h}(M)}
		\,d\mathbb P_{n\kappa/\gamma}(h)
		\\
		&=
		\frac1{Z_{\kappa,\gamma}}
		\int
		F\!\left(\frac{\gamma}{n}(h+c_\kappa)\right)
		e^{-\Theta\langle h+c_\kappa,Q\rangle
			-\frac{na_n\Lambda_0}{\gamma^2}\mu^{\gamma h}(M)}
		\,d\mathbb P(h)
		\\
		&=
		\frac{e^{-\Theta\langle c_\kappa,Q\rangle}}{Z_{\kappa,\gamma}}
		\int
		F\!\left(\frac{\gamma}{n}h+\kappa\right)
		e^{-\Theta\langle h,Q\rangle
			-\frac{na_n\Lambda_0}{\gamma^2}\mu^{\gamma h}(M)}
		\,d\mathbb P(h).
	\end{align*}
	
	Using
	\[
	Z_{\kappa,\gamma}
	=
	e^{-\Theta\langle c_\kappa,Q\rangle}Z_{0,\gamma},
	\]
	we obtain
	\begin{align*}
		\mathbb E_{Q_{\kappa,\gamma}^{\mathrm{sc}}}[F(\hbar)]
		&=
		\frac1{Z_{0,\gamma}}
		\int
		F\!\left(\frac{\gamma}{n}h+\kappa\right)
		e^{-\Theta\langle h,Q\rangle
			-\frac{na_n\Lambda_0}{\gamma^2}\mu^{\gamma h}(M)}
		\,d\mathbb P(h)
		\\
		&=
		\mathbb E_{Q_{0,\gamma}^{\mathrm{sc}}}
		\!\left[
		F(\hbar+\kappa)
		\right].
	\end{align*}
	
	Therefore the law of $\hbar+\kappa$ under
	$Q_{0,\gamma}^{\mathrm{sc}}$ coincides with
	$Q_{\kappa,\gamma}^{\mathrm{sc}}$, i.e.
	\[
	\hbar+\kappa
	\sim
	Q_{\kappa,\gamma}^{\mathrm{sc}}
	\qquad\text{whenever}\qquad
	\hbar\sim Q_{0,\gamma}^{\mathrm{sc}}.
	\]
	
\end{proof}

\subsection{Wick notation for $h$}

Let $ Z \sim \mathcal{N}(0, \varsigma^2) $ be a centered Gaussian random variable.  
For $\alpha\in\N$, the $\alpha$-th \emph{Wick power}  $\wick{Z^\alpha}$ is defined by
\begin{align*}
	\wick{Z^\alpha} = \varsigma^\alpha H_\alpha(\varsigma^{-1}Z)
\end{align*}
where $H_\alpha$ is the \textit{probabilists' Hermite polynomial} of degree $\alpha$
\[
H_\alpha(x) = \alpha! \sum_{j=0}^{\lfloor \alpha/2 \rfloor} \frac{(-1)^j}{j! (\alpha - 2j)!} \frac{x^{\alpha - 2j}}{2^j},
\]
see \cite[Section 1.3]{Jan97}.

For example:
\[
:Z^2: = Z^2 - \E{Z^2} = Z^2 - \varsigma^2
\]
\[
:Z^4: = Z^4 - 6\varsigma^2 Z^2 + 3\varsigma^4.
\]

If $Z$ is not centered, then $\wick{Z^n}$ is defined as
\begin{align*}
	\wick{Z^n}=\wick{\bar Z^n}
\end{align*} 
where $\bar Z =Z-\E[Z]$.

The Wick powers satisfy the orthogonality relation

\begin{align}\label{eq: wick}
	\E[\wick{Z^\alpha} \wick{Y^\beta}] = \alpha! \, \E[Z Y]^\alpha \, \mathbbm{1}_{\{\alpha = \beta\}},
\end{align}
where $Z,Y$ are two centered Gaussian random variables. In particular we have
\begin{align*}
	\E[\wick{Z^\alpha}]=0 .
\end{align*}
We also recall the Wick exponential given by
\begin{align*}
	\wick{e^Z}=\sum_{\alpha=0}^\infty \frac{\wick{Z^\alpha}}{\alpha!},
\end{align*}
which can be expressed by
\begin{align}\label{eq: wickexp}
	\wick{e^Z}=e^{Z-\frac12\varsigma^2}.
\end{align}

We define the Wick powers of the Gaussian field $h$ by considering approximations, cf. \cite[Section 2.3]{lacoin2017semiclassical}.

\begin{proposition}\label{prop: approxwick}
	Let~$(\M,\g)$ be admissible. Let 
	$\seq{h_{\ell}}_{l\in \N}$ be the family of centered Gaussian random fields given in \eqref{eq: white noise trunc} with covariance functions $(k_{\ell})_{\ell\in\N}$. 
	
	
	Then for all $\alpha\in\N$, $\varphi\in L^\infty(M)$,  
	$$\scalar{\wick{h_{\ell}^\alpha}}{\varphi}\to \scalar{\wick{h^\alpha}}{\varphi}$$
	as~${\ell} \to \infty$ in $L^2(\Omega)$. 

\end{proposition}

\begin{proof}

	Let $\alpha\in \N$ and $\varphi\in L^\infty(M)$. Define for ${\ell},m\in\N$
	\begin{align*}
		X_{l,m}=\scalar{\wick{h_{\ell}^\alpha}-\wick{h_m^\alpha}}{\varphi}.
	\end{align*}
	Then by \eqref{eq: wick}
	\begin{align*}
		&\E[|X_{\ell,m}|^2]
		=\iint \E[(\wick{h_{\ell}^\alpha(x)}-\wick{h_m^\alpha(x)})(\wick{h_{\ell}^\alpha(y)}-\wick{h_m^\alpha(y)})]\varphi(x)\varphi(y)\, d\vol(x)\, d\vol(y)\\
		=&\alpha!\iint(k_{\ell}(x,y)^\alpha+k_m(x,y)^\alpha-k_{\ell, m}(x,y)^\alpha-k_{\ell,m}(y,x)^\alpha)\varphi(x)\varphi(y)\, d\vol(x)\, d\vol(y),
	\end{align*}
	where $k_{\ell,m}(x,y)=\E[h_{\ell}(x)h_m(y)]$.
	
	By virtue of \eqref{eq: domination}, $k_{\ell}(x,y)$ is dominated by
	\begin{align*}
		|k_{\ell}(x,y)|^\alpha\leq C(1+(\log(1/d(x,y)))^\alpha) \quad \forall x\neq y
	\end{align*}
	for some constant $C>0$. The right hand side is integrable on $M\times M$, since for $d(x,y)=r$ with $0<r<\eps$ 
	\begin{align*}
		\int_0^\eps r^{n-1}\left(\log\left(\frac1r\right)\right)^\alpha\, dr
		=\int_{t_0}^\infty t^\alpha e^{-nt}\, dt.
	\end{align*}
	Since 
	$k_{\ell}(x,y)\to k(x,y)$ a.e.,
	dominated convergence yields
	\begin{align*}
		\iint k_{\ell}(x,y)^\alpha \varphi(x)\varphi(y)\, d\vol(x)\, d\vol(y)\to \iint k(x,y)^\alpha \varphi(x)\varphi(y)\, d\vol(x)\, d\vol(y)
	\end{align*}
	as ${\ell}\to\infty$. The same reasoning applies to $k_m^\alpha$ and $k_{\ell,m}^\alpha$ and we find
	\begin{align*}
		\E[|X_{\ell,m}|^2]\to0
	\end{align*}
	as ${\ell},m\to\infty$. Since $L^2(\Omega)$ is complete, the sequence converges to a limit in $L^2(\Omega)$ and the limit is denoted by $\scalar{\wick{h^\alpha}}{\varphi}$.

\end{proof}

\begin{proposition}\label{prop: approxm}
	Let $M$ be admissible. Let $h,(h_{\ell})_{l\in\N}$ be as in Proposition \ref{prop: approxwick} and $0<\gamma<\sqrt{2n}$. Then
	for all $\varphi\in L^\infty(M)$ we have 
	\begin{align*}
		\sum_{\alpha=0}^\infty\frac{\gamma^\alpha}{\alpha!}\scalar{\wick{h^\alpha}}{\varphi}=\int_M \varphi(x)\, d\mu^{\gamma  h}(x),
	\end{align*}
	where the above sum converges in $L^2(\Omega)$.
\end{proposition}
\begin{proof}
	By virtue of Theorem \ref{t:WhiteNoiseApproxGMC}, we know that for $0<\gamma<\sqrt{2n}$ 
	\begin{align*}
		\lim_{l\to\infty}\int \varphi(x)e^{\gamma h_{\ell}(x)-\frac{\gamma^2}2k_{\ell}(x,x)}\, d\vol(x)=\int \varphi(x)\, d\mu^{\gamma h}(x).
	\end{align*}
	Applying  \eqref{eq: wickexp} to the Gaussian random field $h_{\ell}(x)$, we obtain for all $x\in M$
	\begin{align*}
		e^{\gamma h_{\ell}(x)-\frac{\gamma^2}2k_{\ell}(x,x)}=\sum_{\alpha=0}^\infty \frac{\gamma^\alpha}{\alpha!}\wick{h_{\ell}(x)^\alpha}\ .
	\end{align*}
	
	Testing against $\varphi$, we find after applying Fubini and \eqref{eq: domination}
	\begin{align*}
		\sum_{\alpha=0}^N\frac{\gamma^\alpha}{\alpha!}\scalar{\wick{h_{\ell}^\alpha}}{\varphi}\to\sum_{\alpha=0}^\infty\frac{\gamma^\alpha}{\alpha!}\scalar{\wick{h_{\ell}^\alpha}}{\varphi}
	\end{align*}
	as $N\to\infty$ in $L^2(\Omega)$ for each fixed ${\ell}$. Indeed,
	\begin{align*}
		\E\left[\big(\sum_{\alpha=0}^N\frac{\gamma^\alpha}{\alpha!}\scalar{\wick{h_{\ell}^\alpha}}{\varphi}\big)^2\right]=&\sum_{\alpha=0}^N\sum_{\beta=0}^N\frac{\gamma^\alpha\gamma^\beta}{\alpha!\beta!}\E\left[\scalar{\wick{h_{\ell}^\alpha}}{\varphi}\scalar{\wick{h_{\ell}^\beta}}{\varphi}\right]\\
		=&\sum_{\alpha=0}^N\sum_{\beta=0}^N\frac{\gamma^\alpha\gamma^\beta}{\alpha!\beta!}\iint \varphi(x) \varphi(y)\E\left[\wick{h_{\ell}^\alpha}\wick{h_{\ell}^\beta}\right]d\vol(x)d \vol(y)\\
		=&\sum_{\alpha=0}^N\frac{\gamma^{2\alpha}}{\alpha!}\iint \varphi(x) \varphi(y)k_{\ell}(x,y)^\alpha d\vol(x)d \vol(y),
	\end{align*}
	where we applied Fubini's theorem in the second line and \eqref{eq: wick} in the third line. Since $\varphi\in L^\infty(M)$ and $k_{\ell}$ satisfies \eqref{eq: domination}, the right hand side can be bounded by 
	\begin{align*}
		\iint \exp(\gamma^2 C(1+|\log d(x,y)|))\, d\vol(x)\, d\vol(y),
	\end{align*}
	which is finite since $\gamma^2<2n$.

	It remains to show that
	\begin{align*}
		\lim_{\ell\to\infty}\sum_{\alpha=0}^\infty\frac{\gamma^\alpha}{\alpha!}\scalar{\wick{h_{\ell}^\alpha}}{\varphi}=\sum_{\alpha=0}^\infty\frac{\gamma^\alpha}{\alpha!}\scalar{\wick{h^\alpha}}{\varphi}.
	\end{align*}
	For each fixed $\alpha$ we have by Proposition \ref{prop: approxwick}
	\begin{align*}
		\scalar{\wick{h_{\ell}^\alpha}}{\varphi}\to \scalar{\wick{h^\alpha}}{\varphi}
	\end{align*}
	as ${\ell}\to\infty$ in $L^2(\Omega)$. Since 
	\begin{align*}
		&\sum_{\alpha=0}^\infty\sup_{\ell}\|\frac{\gamma^\alpha}{\alpha!}\scalar{\wick{h_{\ell}^\alpha}}{\varphi}\|_{L^2(\Omega)}\\
		&\leq C\sum_{\alpha=0}^\infty \frac{\gamma^\alpha}{\alpha!}\sup_{\ell}\sqrt{\iint|k_{\ell}(x,y)|^\alpha d\vol(x)d\vol(y)},
	\end{align*}
	the series of $L^2$-random variables is uniformly absolutely convergent in $L^2(\Omega)$ by virtue of \eqref{eq: domination}. Hence, we may interchange limit and sum, which concludes the proof.
\end{proof}

\subsection{Massive fields}\label{sec: mass}

In the following we define the massive fields on an admissible manifold $(M,g)$.
For $m\geq0$, let 
\begin{align*}
	\mathsf p_{m}=\mathsf p+m.
\end{align*}
For $m>0$, the operator $\mathsf p_m$ is strictly positive, self-adjoint and elliptic on $L^2(M)$. Hence $\mathsf p_{m}^{-1}$ is a bounded, self-adjoint, positive operator on $L^2$ and maps $H^{s}$ onto $H^{s+n}$ for every $s\in\R$. 
By spectral decomposition of $\mathsf p$, using the $L^2$-orthonormal eigenbasis $\{\psi_j\}_{j\in \N}$ of $\mathsf p$ with eigenvalues $(\nu_j)_{j\in\N}$, $\mathsf p_m^{-1}$ admits a kernel 
\begin{align*}
	k_{m}(x,y)=\sum_{j=0}^\infty \frac{\psi_j(x)\psi_j(y)}{\nu_j+m},
\end{align*}
where $\psi_0=|M|^{-\frac12}$ and $\nu_0=0$ and 
the convergence of the series is understood in $L^2(M\times M)$, cf.\ \cite[Lemma 2.15]{Del24}. This kernel is smooth outside the diagonal with logarithmic divergence close to it.

For $m\geq0$, we consider the grounded operator $\mathring {\mathsf p}_m^{-1}$ on $L^2(M)$ given by
\begin{align*}
	\mathring {\mathsf p}_m^{-1}(\phi)={\mathsf p}_m^{-1}(\mathring \phi),\quad \mathring \phi=\phi-\langle \phi\rangle.
\end{align*}
This operator again admits a kernel given by
\begin{align*}
	\mathring k_{m}(x,y)=\sum_{j=1}^\infty \frac{\psi_j(x)\psi_j(y)}{\nu_j+m}
\end{align*}
and for $m>0$ it satisfies
\begin{align*}
	\mathring k_{m}(x,y)=k_m(x,y)-\frac1{m|M|}.
\end{align*}
Then,  the associated positive-definite bilinear form is 
\begin{align*}
	\mathring{	\mathcal K}_{m}(\phi,\eta)=&\scalar{\phi}{\mathring{\mathsf p}^{-1}_m\eta}=\sum_{j=1}^\infty \frac{\scalar{\phi}{\psi_j}\scalar{\eta}{\psi_j}}{\nu_j+m}.
\end{align*}
%
Next, we define a Gaussian field associated to the kernel $\mathring k_m$.

\begin{definition} A massive co-polyharmonic Gaussian field on $(M,g)$ with mass $m>0$ is a centered Gaussian random variable $\tilde h_0$ on $\mathring{\mathfrak D}'$ with covariance
	\begin{equation}
		\label{eq:covariance-chp-massive-grounded}
		\mathbb{E} \Big[\langle \tilde h_0|\phi\rangle \langle \tilde h_0|\eta\rangle\Big]=\mathring{\mathcal K}_{m}(\phi,\eta)
		\quad\qquad \forall \phi,\eta\in \mathring{\mathfrak D} ,
	\end{equation}
	where
	\begin{align*}
		\mathring{\mathcal K}_{m}(\phi,\eta)
		=&\sum_{j=1}^\infty \frac{\scalar{\phi}{\psi_j}\scalar{\eta}{\psi_j}}{\nu_j+m}.
	\end{align*}

\end{definition}

\begin{theorem}
	For every admissible manifold~$(\M,g)$ there exists a massive co-polyharmonic Gaussian field with mass $m>0$, unique in distribution.

\end{theorem}

\begin{proof}
	The space $\mathring{\mathfrak D}$ endowed with its usual Fr\'echet topology is a nuclear space \cite[p.\ 55]{Gro66}. We will show that
	\begin{align*}
		\chi(\phi)=\exp(-\frac12\mathring{\mathcal K}_{m}(\phi,\phi))
	\end{align*}
	is a continuous positive-definite function on $\mathring{\mathfrak D}$ with $\chi(0)=1$. Then the Bochner-Minlos theorem \cite[p.\ 410]{VakTarCho87} asserts that $\chi$ is the characteristic function of a unique probability measure on the dual space $\mathring{\mathfrak D'}$ of $\mathring{\mathfrak D}$. This probability measure is precisely the law of the centered Gaussian field with covariance $\mathsf p_{m}^{-1}$.
	
	Let $\phi\in \mathring{\mathfrak D}$. Then by Cauchy-Schwarz
	\begin{align*}
		|\scalar{\phi}{\psi_j}|\leq (\nu_j+m)^{-1}\|\mathsf p_{m} \phi\|_{L^2}.
	\end{align*}
	Since $\nu_j\sim cj$, the series
	\begin{align*}
		\sum_{j=1}^\infty (\nu_j+m)^{-3}<\infty
	\end{align*}	
	converges. Hence
	\begin{align*}
		|\mathring{\mathcal K}_{m}(\phi,\phi)|\leq\sum_{j=1}^\infty (\nu_j+m)^{-3} \|\mathsf p_{m}\phi\|_{L^2}^2
	\end{align*}
	and by polarization this proves $\mathring{\mathcal K}_{m}$ is a continuous bilinear form on $\mathring{\mathfrak D}\times \mathring{\mathfrak D}$.
	
	Therefore, $\chi$ is continuous on $\mathring{\mathfrak D}$ with $\chi(0)=1$ and by \cite[Proposition 2.4]{LodSheSunWat16} it is positive definite.
\end{proof}

The next theorem gives an explicit construction of $\tilde h_0$ by using the eigenbasis  $(\psi_j)_{j\in\N}$. In particular, it turns out that 
$\tilde h_0 \in  H^{-\varepsilon}$ $\mathbb{P}$-a.s.\ for every $\varepsilon>0$.
\begin{theorem}\label{thm: massive}
	Let $(\psi_j)_{j\in\N}$ be a complete $L^2$-orthonormal system consisting of eigenfunctions of~$\mathsf p$, each with corresponding eigenvalue~$\nu_j$. Let $(\xi_j)_{j\in \N}$ be a sequence  of independent and identically distributed standard Gaussian variables. For each~${\ell} \in \mathbb{N}$, we define the random test function
	\begin{align*}
		\tilde h_{\ell}(x)\coloneqq  \sum_{j=1}^{\ell} \frac{\psi_j(x)}{\sqrt{\nu_j+m}}\, \xi_j, \qquad x \in M \fstop
	\end{align*}
	Then, the covariance of the random field $\tilde h_{\ell}$ is given by:
	\begin{equation*}
		\tilde	k_{\ell}(x,y)\coloneqq  {\E}\Big[\tilde h_{\ell}(x)\, \tilde h_{\ell}(y)\Big]=\sum_{j=1}^{\ell} \frac{\psi_j(x)\, \psi_j(y)}{\nu_j+m}\comma \qquad x, y \in M \fstop
	\end{equation*}
	
	Moreover, for all $\varepsilon > 0$, the field~$\tilde h_{\ell}$, regarded as a random element of ${H}^{-\varepsilon}$, converges as~${\ell} \to \infty$ to a massive co-polyharmonic Gaussian field $\tilde h_0$ in $L^{2}(\mathbb{P})$ and $\mathbb{P}$-a.s.
	In particular, $\tilde h \in  H^{-\varepsilon}$ $\mathbb{P}$-a.s.. 
\end{theorem}

\begin{proof}
	The proof goes along the lines of the proof in \cite[Proposition 3.9]{Del24}.
\end{proof}

Since, we employ the augmented field, we introduce the augmented massive field in the next definition.
\begin{definition} An \emph{augmented massive (co-polyharmonic) Gaussian field} $\tilde h$ on $(M,g)$ with mass $m>0$ is a centered Gaussian random variable $\tilde h$ on ${\mathfrak D}$ with covariance
	\begin{equation}
		\label{eq:covariance-chp-massive-grounded1}
		\mathbb{E} \Big[\langle \tilde h|\phi\rangle\cdot \langle \tilde h|\eta\rangle\Big]={\mathcal K}_{m}(\phi,\eta)
		\quad\qquad \forall \phi,\eta\in {\mathfrak D} ,
	\end{equation}
	where
	\begin{align*}
		{\mathcal K}_{m}(\phi,\eta)
		=&\sum_{j=1}^\infty \frac{\scalar{\phi}{\psi_j}\scalar{\eta}{\psi_j}}{\nu_j+m}+\sigma^2\scalar{\phi}{1}\scalar{\eta}{1}.
	\end{align*}

\end{definition}

In particular, if $\tilde h_0$ is an $m$-massive co-polyharmonic field and $\xi\sim\mathcal N(0,1)$ is independent, then
\begin{align*}
	\tilde h=\tilde h_0+\sigma\xi
\end{align*}
is an augmented $m$-massive co-polyharmonic field. Moreover, if $m=0$ the augmented $m$-massive field coincides with the augmented field.

\section{Proof of Theorem \ref{thm: main}}\label{sec: main}

\subsection{Analysis of partition function}
As a first step we consider the normalization or partition function of the Polyakov-Liouville measure ${\boldsymbol{\nu}}_{\kappa,\gamma}$ for $\kappa=0$. Recall that by Corollary \ref{cor: liouville} there exists a unique smooth solution $u=u_0$ to \eqref{eq: liouville} with $\int u\, d\vol=0$ and
$\Lambda=\Lambda_0=-Q(M)/\int e^{nu}\, d\vol$.  Let $\Theta=na_n/\gamma$ and $m=na_n\Lambda/\gamma^2$. Note that by Lemma \ref{lemma: finite} the partition function is finite and is given by
\begin{align*}
	Z_\gamma=\E\left[\exp(-\frac{na_n}{\gamma}\scalar{h}{Q}-\frac{na_n}{\gamma^2}\Lambda\mu^{\gamma h}(M))\right],
\end{align*}
where the expectation is taken with respect to the law $\P$ of the augmented field $h$.

\begin{proposition}\label{prop: pf}
	Let $0<\gamma<\sqrt{2n}$. Then
	\begin{align*}	Z_\gamma=e^{\gamma^{-2}F(\Lambda)}\E\left[\exp\left(-\frac{na_n\Lambda}{\gamma^2}\scalar{e^{nu}}{\mu^{\gamma h}-1-\gamma h}\right)\right],
	\end{align*}
	where
	\begin{align*}
		F(\Lambda)=-\frac{n^2}2{\mathfrak p}(u,u)-na_n\scalar{\Lambda e^{nu}+nQu}{1}.
	\end{align*}
\end{proposition}
\begin{proof}
	\emph{Step 1: Insert the solution $u$.} We define
	\begin{align*}
		h^u:=h-\frac{n}{\gamma}u.
	\end{align*}
	Then
	\begin{align*} Z_\gamma=&\E\Big[\exp(-\frac{na_n}{\gamma}\scalar{h}{Q}\exp\Big(-\frac{na_n\Lambda}{\gamma^2}\scalar{e^{nu}}{1+\gamma h^u}\Big)\\
		&\quad \cdot\exp\Big(-\frac{na_n\Lambda}{\gamma^2}\mu^{\gamma h}(M)+\frac{na_n\Lambda}{\gamma^2}\scalar{e^{nu}}{1+\gamma h^u}\Big)\Big]\\
		=&\exp\Big(-\frac{na_n\Lambda}{\gamma^2}\scalar{e^{nu}}{1-nu}\Big)\\ &\cdot\E\Big[\exp(-\frac{na_n}{\gamma}\scalar{h}{Q}\exp\Big(-\frac{na_n\Lambda}{\gamma}\scalar{e^{nu}}{h}\Big)\\
		&\quad\cdot\exp\Big(-\frac{na_n\Lambda}{\gamma^2}\mu^{\gamma h}(M)+\frac{na_n\Lambda}{\gamma^2}\scalar{e^{nu}}{1+\gamma h^u}\Big)\Big]\\
	\end{align*}
	\emph{Step 2: Employ Girsanov.} Since $u$ solves \eqref{eq: liouville}, the first two terms in the expectation can be summarized as
	\begin{align*}
		\exp\left(-\frac{na_n}\gamma\scalar{h}{Q+\Lambda e^{nu}}\right)=\exp\left(\frac{n}\gamma\scalar{h}{\mathsf p u}\right).
	\end{align*}
	We recall that $\langle u\rangle=0$ and hence the expression
	\begin{align*}
		\exp\left(\frac{n}\gamma\scalar{h}{\mathsf p u}-\frac{n^2}{2\gamma^2}\mathfrak p(u,u)\right)
	\end{align*}
	is by \eqref{eq: girsanov0} the density of the law of the shifted field $
	h+\frac{n}\gamma u$
	with respect to the law of $h$. 
	With this observation we rewrite
	\begin{align*}
		Z_\gamma=&\exp\Big(-\frac{na_n\Lambda}{\gamma^2}\scalar{e^{nu}}{1-nu}+\frac{n^2}{2\gamma^2}\mathfrak p(u,u)\Big)\\
		&\cdot\E\Big[\exp\Big(-\frac{na_n}{\gamma}\scalar{h}{Q+\Lambda e^{nu}}-\frac{n^2}{2\gamma^2}\mathfrak p(u,u)\Big)\\
		&\quad\cdot\exp\Big(-\frac{na_n\Lambda}{\gamma^2}\mu^{\gamma h}(M)+\frac{na_n\Lambda}{\gamma^2}\scalar{e^{nu}}{1+\gamma h^u}\Big)\Big]\\
		=&\exp\Big(-\frac{na_n\Lambda}{\gamma^2}\scalar{e^{nu}}{1-nu}+\frac{n^2}{2\gamma^2}\mathfrak p(u,u)\Big)\\
		&\cdot\E\Big[\exp\Big(-\frac{na_n\Lambda}{\gamma^2}\mu^{\gamma h+nu}(M)+\frac{na_n\Lambda}{\gamma^2}\scalar{e^{nu}}{1+\gamma h}\Big)\Big]\\
		=&\exp\Big(-\frac{na_n\Lambda}{\gamma^2}\scalar{e^{nu}}{1-nu}+\frac{n^2}{2\gamma^2}\mathfrak p(u,u)\Big)\\
		&\cdot\E\Big[\exp\Big(-\frac{na_n\Lambda}{\gamma^2}\scalar{e^{nu}}{\mu^{\gamma h}-1-\gamma h}\Big)\Big],
	\end{align*} 
	where we used \eqref{e:liouville-cm-shift}.
	
	\emph{Step 4: Compute the prefactor.} We compute
	\begin{align*}
		\scalar{\Lambda e^{nu}}{1-nu}=&\scalar{\Lambda e^{nu}+nQu}{1}-\scalar{\Lambda e^{nu}+Q}{nu}\\
		=&\frac{n}{a_n}\mathfrak p(u,u)+\scalar{\Lambda e^{nu}+nQu}{1}.
	\end{align*}
	Consequently,
	\begin{align*}
		&-\frac{na_n}{\gamma^2}\scalar{\Lambda e^{nu}}{1-nu}+\frac{n^2}{2\gamma^2}\mathfrak p(u,u)	\\
		&=-\frac{n^2}{2\gamma^2}\mathfrak p(u,u)-\frac{na_n}{\gamma^2}\scalar{\Lambda e^{nu}+nQu}{1}\\
		&=\gamma^{-2}F(\Lambda),
	\end{align*}
	which proves the claim.
\end{proof}

\subsection{Partition function as $\gamma\to0$}

In the following we want to study the behavior of
$$e^{-\gamma^{-2}F(\Lambda)}Z_\gamma=\E\Big[\exp\Big(-\frac{na_n\Lambda}{\gamma^2}\scalar{e^{nu}}{\mu^{\gamma h}-1-\gamma h}\Big)\Big]$$
as $\gamma\to0$.	For this we prove the following proposition.
\begin{proposition}\label{prop: convergence}
	For all positive bounded functions $\varphi\colon M\to\R$ we have
	\begin{align*}
		\lim_{\gamma\to0}\E[\exp(-\frac1{\gamma^2}\scalar{\varphi}{\mu^{\gamma h}-1-\gamma h})]=\E[\exp(-\frac12\scalar{\varphi}{\wick{h^2}})].
	\end{align*}
\end{proposition}

\begin{proof}
	Following the proof of \cite[Lemma 3.6]{lacoin2017semiclassical} we define
	\begin{align*}
		H_\gamma =\frac1{\gamma^2}\scalar{\mu^{\gamma h}-1-\gamma h-\frac{\gamma^2}2\wick{h^2}}{\varphi}
	\end{align*}
	and show that it converges towards $0$ in $L^2(\P)$ as $\gamma\to0$. By virtue of Proposition \ref{prop: approxm} we write
	\begin{align*}
		H_\gamma =&\frac1{\gamma^2}\left(\sum_{\alpha=0}^\infty\frac{\gamma^\alpha}{\alpha!}\scalar{\wick{h^\alpha}}{\varphi} -\scalar{1}{\varphi}-\scalar{\gamma h}{\varphi}-\scalar{\frac{\gamma^2}2\wick{h^2}}{\varphi}\right)\\
		=&\frac1{\gamma^2}\sum_{\alpha=3}^\infty\frac{\gamma^\alpha}{\alpha!}\scalar{\wick{h^\alpha}}{\varphi}.
	\end{align*}
	Obviously we have that $\E[H_\gamma]=0$ by definition of the Wick powers.
	Concerning its variance we simply compute
	\begin{align*}
		\E[(\frac1{\gamma^2}\sum_{\alpha=3}^\infty\frac{\gamma^\alpha}{\alpha!}\scalar{\wick{h^\alpha}}{\varphi})^2]
		=&\frac1{\gamma^4}\sum_{\alpha,\beta\geq 3}\frac{\gamma^{\alpha}\gamma^\beta}{\alpha!\beta!}\E[\scalar{\wick{h^\alpha}}{\varphi}\scalar{\wick{h^\beta}}{\varphi}]\\
		=&\frac1{\gamma^4}\sum_{\alpha\geq 3}\frac{\gamma^{2\alpha}}{\alpha!}\iint k(x,y)^\alpha \varphi(x)\varphi(y)\, d\vol(x)\, d\vol(y),
	\end{align*}
	where we used Proposition \ref{prop: approxm}, Proposition \ref{prop: approxwick} and \eqref{eq: wick} in the second line. Due to \eqref{eq: log}, the last expression is of order $O(\gamma^2)$ since
	\begin{align*}
		&\frac1{\gamma^4}\sum_{\alpha\geq 3}\frac{\gamma^{2\alpha}}{\alpha!}\iint k(x,y)^\alpha \varphi(x)\varphi(y)\, d\vol(x)\, d\vol(y)\\
		&=\frac{\gamma^2}6\iint k(x,y)^3 \varphi(x)\varphi(y)\, d\vol(x)\, d\vol(y)+O(\gamma^4)
	\end{align*} 
	and we obtain that $H_\gamma\to0$ in $L^2(\P)$.  
	
	In order to prove convergence of $\exp(-\frac1{\gamma^2}\scalar{\varphi}{\mu^{\gamma h}-1-\gamma h})$ in $L^1$ we apply the mean value theorem to the exponential function and Cauchy-Schwarz which leads to
	\begin{align*}
		&\E[|e^{-\frac1{\gamma^2}\scalar{\varphi}{\mu^{\gamma h}-1-\gamma h}}-e^{-\frac12\scalar{\varphi}{\wick{h^2}}}|]\\
		&\leq 	\E[|e^{-\frac1{\gamma^2}\scalar{\varphi}{\mu^{\gamma h}-1-\gamma h}}+e^{-\frac12\scalar{\varphi}{\wick{h^2}}}|^2]^{1/2}	\E[|H_\gamma|^2]^{1/2}.
	\end{align*}
	Since the first expectation on the right hand side is uniformly bounded in $\gamma$ by Lemma \ref{lem: wickfinite} and Lemma \ref{lem: bounded} below and $H_\gamma\to0$ in $L^2(\P)$, we conclude the argument. 
\end{proof}

\begin{lemma}\label{lem: wickfinite}
	Assume $u$ is smooth. 
	\begin{enumerate}[$(i)$]
		\item
		For all $c\geq0$
		\begin{align*}
			\tilde Z_u=\E\Big[\exp\Big(-c\scalar{e^{nu}}{ \wick{h^2}}\Big)\Big]<\infty.
		\end{align*}
		\item 
		Moreover under the measure 
		$$ \tilde\P_u:=\tilde Z_u^{-1}\exp\Big(-c\scalar{e^{nu}}{ \wick{h^2}}\Big)\P,$$ 
		the field $h$ is the augmented massive field with mass $2c$ under the metric $\hat g$.
	\end{enumerate}
\end{lemma}

\begin{proof}
	$(i)$ Let $\hat g=e^{2u}g$ and note that $\mathsf p_{\hat g}=e^{-nu}\mathsf p$ for all $u\in C^\infty(M)$. Let $\nu_j$ be eigenvalues of $\mathsf p_{\hat g}$ and $\psi_j$ be an ONB associated to $\nu_j$ in $L^2(\vol_{\hat g})$ and define
	\begin{align*}
		\tilde h{:=}\sum_{j=1}^\infty\frac1{\sqrt{\nu_j}}\psi_j\xi_j,
	\end{align*}
	where $(\xi_j)_{j=1}^\infty$ are iid standard normally distributed random variables. With this we compute
	\begin{align*}
		\scalar{e^{nu}}{ \wick{\tilde h^2}}=&\lim_{l\to\infty}\scalar{e^{nu}}{ \Big(\sum_{j=1}^{\ell}\frac1{\sqrt{\nu_j}}\psi_j\xi_j\Big)^2-\sum_{j=1}^{\ell}\frac1{\nu_j}\psi_j^2}\\
		=&\sum_{j=1}^\infty\frac1{\nu_j}(\xi_j^2-1).
	\end{align*}
	Inserting this into $\tilde Z_u$
	\begin{align*}
		\tilde Z_u=\prod_j\E[\exp(-\frac{c}{\nu_j}\xi_j^2)]e^{c/\nu_j}=\prod_j\sqrt{\frac{\nu_j}{\nu_j+2c}}e^{c/\nu_j}<\infty.
	\end{align*}
	The last equality holds because $\xi_j^2\sim \chi_1^2$. The expression is finite for any $c>-\nu_1/2$, which can be checked by considering $\sum_j-\frac12\log(1+\frac{2c}{\nu_j})+\frac{c}{\nu_j}$ and using that $\sum\frac1{\nu_j^2}<\infty$ by Weyl's law \eqref{eq: weyl}.
	
	$(ii)$ The field $\tilde h$ is Gaussian under $\tilde\P_u$ since for the Laplace transform
	\begin{align*}
		\tilde \E_u[e^{t\tilde h(x)}]=\tilde Z_c^{-1}\prod_j\E\Big[\exp\Big(t\frac1{\sqrt{\nu_j}}\psi_j(x)\xi_j\Big)\exp\Big(-\frac{c}{\nu_j}\xi_j^2\Big)\Big]e^{\frac{c}{\nu_j}},
	\end{align*}
	where the expectation can be explicitly computed
	\begin{align*}
		\E\Big[\exp\Big(t\frac1{\sqrt{\nu_j}}\psi_j(x)\xi_j\Big)\exp\Big(-\frac{c}{\nu_j}\xi_j^2\Big)\Big]=\frac1{\sqrt{1+\frac{2c}{\nu_j}}}\exp\left(\frac{(\frac{\psi_j(x)}{\sqrt{\nu_j}})^2t^2}{2(1+\frac{2c}{\nu_j})}\right),
	\end{align*}
	which is the Laplace transform of a shifted Gaussian random variable.
	
	So it remains to compute the expectation and covariance. The expectation clearly vanishes by symmetry
	\begin{align*}
		\tilde \E_u[\tilde h(x)]=\tilde Z_u^{-1}\sum_k\frac{\psi_k(x)}{\sqrt{\nu_k}}\E[\xi_k\exp(-c\sum_j(\xi_j^2-1)/\nu_j)]=0.
	\end{align*}
	
	For the covariance
	\begin{align*}
		\tilde\E_u[\tilde h(x)\tilde h(y)]=\tilde Z_u^{-1}\sum_j\frac{\psi_j(x)\psi_j(y)}{\nu_j}\E[\xi_j^2 \exp(-c\sum_k(\xi_k^2-1)/\nu_k)].
	\end{align*}
	The expectation is 
	\begin{align*}
		&\E[\xi_j^2 \exp(-c\sum_k(\xi_k^2-1)/\nu_k)]\\
		&=\E[\xi_j^2 \exp(-c(\xi_j^2-1)/\nu_j)]\prod_{k\neq j}\E[\exp(-c\sum_k(\xi_k^2-1)/\nu_k)].
	\end{align*}
	We note that for the first term
	\begin{align*}
		\E[\xi_j^2 \exp(-c(\xi_j^2-1)/\nu_j)]=e^{c/\nu_j}\left(\frac{\nu_j+2c}{\nu_j}\right)^{-3/2}
	\end{align*}
	by direct computation and for the second
	\begin{align*}
		\prod_{k\neq j}\E[\exp(-c\sum_k(\xi_k^2-1)/\nu_k)]=\tilde Z_u\sqrt{\frac{\nu_j+2c}{\nu_j}}e^{-c/\nu_j}.
	\end{align*}
	Together this gives us for the covariance
	\begin{align*}
		\tilde\E_u[\tilde h(x)\tilde h(y)]=\sum_{j=1}^\infty\frac{\psi_j(x)\psi_j(y)}{\nu_j+2c},
	\end{align*}
	i.e. $\tilde h$ is under $\tilde\P_u$ a centered Gaussian with covariance given by $(\mathring{\mathsf p}_{\hat g,2c})^{-1}$ by \eqref{eq:covariance-chp-massive-grounded}. 
	
	By \cite[Proposition 2.22(ii)]{Del24} and Lemma \ref{lem: h and h*} we have
	\begin{align*}
		e^{nu}h\overset{d}{=}\tilde h+\sigma\xi e^{nu},
	\end{align*}
	where $\scalar{e^{nu}h}{\phi}=\scalar{h}{e^{nu}\phi}$ denotes the dual pairing with respect to the metric $\hat g$ and $\xi$ is a standard normal distributed random variable independent from $\tilde h$. In particular, by the computation for $\tilde h$, we have
	\begin{align*}
		\tilde \E_u[\scalar{e^{nu}h}{\phi}^2]=\sum_{j=1}^\infty\frac{\scalar{\phi}{\psi_j}_{\hat g}^2}{\nu_j+2c}+\sigma^2\scalar{\phi}{1}_{\hat g}^2.
	\end{align*}
	Consequently, $h$ is then the augmented massive field  with mass $2c$ under the metric $\hat g$ by \eqref{eq:covariance-chp-massive-grounded1}.

\end{proof}

\begin{lemma}\label{lem: bounded}
	For all positive bounded functions $\varphi\colon M\to\R$ we have
	\begin{align*}
		\sup_{0<\gamma\leq 1}\E[\exp(-\frac1{\gamma^2}\scalar{\varphi}{\mu^{\gamma h}-1-\gamma h})]<\infty.
	\end{align*}
\end{lemma}

\begin{proof}
	Let us assume that w.l.o.g.\ $\|\varphi\|\leq 1$. 
	We adapt the proof of Lemma 3.6 in \cite{lacoin2017semiclassical} to our setting.
	
	\emph{Step 1: }For this we consider the approximation $(h_{\ell})$ given by \eqref{eq: white noise trunc} with kernels $k_{\ell}(x,y)$ and define
	\begin{align*}
		R_{\ell}(x)=e^{\gamma h_{\ell}(x)-\frac{\gamma^2}2k_{\ell}(x,x)}-1-\gamma h_{\ell}(x)-\frac{\gamma^2}2\wick{h_{\ell}(x)^2}.
	\end{align*}
	
	We consider the set 
	\begin{align*}
		\mathcal A=\left\{\int_M |R_{\ell}(x)|\mathbbm{1}_{\{|h_{\ell}(x)|>|\log l|^2\}}\, d\vol(x)\geq \gamma^3\right\}.
	\end{align*}
	
	On the set $\mathcal A$ we use \eqref{eq: domination} and obtain
	\begin{align*}
		-\gamma^{-2}( e^{\gamma h_{\ell}(x)-\frac{\gamma^2}2k_{\ell}(x,x)}-1-\gamma h_{\ell}(x))\leq \frac12k_{\ell}(x,x)\leq C\log \ell
	\end{align*}
	for some positive constant $C>0$.
	In particular
	\begin{align*}
		\E\left[\exp\left(-\frac1{\gamma^2}\{\mu^{\gamma h_{\ell}}(\varphi)-\scalar{\varphi}{1}-\scalar{\gamma h_{\ell}}{\varphi}\}\right)\mathbbm 1_{\mathcal A}\right]\leq \int_M e^{\frac12k_{\ell}(x,x)}\, d\vol(x)\P[\mathcal A].
	\end{align*}
	We estimate $\P[\mathcal A]$ by employing Markov's inequality: For this we check by the usual calculations for Gaussian random variables
	\begin{align*}
		\E\left[|R_{\ell}(x)|\mathbbm 1_{\{|h_{\ell}(x)|>|\log l|^2\}}\right]\leq e^{-c|\log \ell|^3}
	\end{align*}
	for some positive constant $c>0$.
	Then
	\begin{align*}
		\P[\mathcal A]\leq e^{-c|\log \ell|^3}\gamma^{-3}
	\end{align*}
	and consequently
	\begin{align*}
		\int_M e^{\frac12k_{\ell}(x,x)}\, d\vol(x)\P[\mathcal A]\leq e^{-c|\log \ell|^3}\gamma^{-3}\ell^C.
	\end{align*}
	
	Then for 
	\begin{align*}
		\ell(\gamma)=\lceil e^{\gamma^{-1/6}}\rceil
	\end{align*}
	we find that
	\begin{align*}
		\sup_{0<\gamma\leq 1}\E\left[\exp\left(-\frac1{\gamma^2}\{\mu^{\gamma h_{\ell(\gamma)}}(\varphi)-\scalar{\varphi}{1}-\scalar{\gamma h_{\ell(\gamma)}}{\varphi}\}\right)\mathbbm 1_{\mathcal A}\right]<\infty.
	\end{align*}

	On the complement of $\mathcal A$ we use the estimate
	\begin{align*}
		R_{\ell}(x)=&R_{\ell}(x)\mathbbm1_{\{|h_{\ell}(x)|>\log \ell|^2\}}+R_{\ell}(x)\mathbbm1_{\{|h_{\ell}(x)|\leq \log \ell|^2\}}\\
		\geq &R_{\ell}(x)\mathbbm1_{\{|h_{\ell}(x)|>\log \ell|^2\}}-\gamma^3|\log \ell|^6/6,
	\end{align*}
	where we used  $e^u-1-u-u^2/2\geq u^3/6$ in the second line and obtain:
	\begin{align*}
		&\E\left[\exp\left(-\frac1{\gamma^2}\{\mu^{\gamma h_{\ell}}(\varphi)-\scalar{\varphi}{1}-\scalar{\gamma h_{\ell}}{\varphi}\}\right)\mathbbm 1_{\mathcal A^c}\right]\\
		=&\E\left[\exp\left(-\frac1{\gamma^2}\int_M \varphi(x)(R_{\ell}(x)+\frac{\gamma^2}2\wick{h_{\ell}(x)^2})\, d\vol(x)\right)\mathbbm 1_{\mathcal A^c}\right]\\
		\leq &\E\left[\exp\left(-\frac1{\gamma^2}\int_M \varphi(x)(R_{\ell}(x)\mathbbm1_{\{|h_{\ell}(x)|>\log \ell|^2\}}-\gamma^3|\log \ell|^6/6+\frac{\gamma^2}2\wick{h_{\ell}(x)^2})\, d\vol(x)\right)\mathbbm 1_{\mathcal A^c}\right]\\
		\leq  &\E\left[\exp\left(\gamma+\gamma|\log \ell|^6/6-\frac{1}2\int_M\varphi(x)\wick{h_{\ell}(x)^2}\, d\vol(x)\right)\mathbbm 1_{\mathcal A^c}\right].
	\end{align*}
	With the same choice of ${\ell}(\gamma)$ as above we see that the last expression can be bounded uniformly for $\gamma\in (0,1]$ since thanks to Lemma \ref{lem: wickfinite}
	\begin{align*}
		\sup_{\ell}\E\left[\exp(-\frac12\int \varphi(x)\wick{h_{\ell}(x)^2}\, d\vol(x))\right]<\infty.
	\end{align*}

	Hence we bounded
	\begin{align*}
		\E\left[\exp\left(-\frac1{\gamma^2}\{\mu^{\gamma h_{\ell}}(\varphi)-\scalar{\varphi}{1}-\scalar{\gamma h_{\ell}}{1}\}\right)\right]\leq C'
	\end{align*}
	uniformly in $\gamma$ for ${\ell}(\gamma)$.
	
	\emph{Step 2: } We still need to remove the cutoff.
	In order to remove the cut-off, we define the set 
	\begin{align*}
		\mathcal B=\{\scalar{h-h_{\ell}}{\varphi}\geq \gamma^2\}\cup \{\mu_{\ell}-\mu\geq \gamma^3\},
	\end{align*}
	where 
	\begin{align*}
		\mu_{\ell}=\mu^{\gamma h_{\ell}}(\varphi),\quad \mu=\mu^{\gamma h}(\varphi).
	\end{align*}
	On the complement of $\mathcal B$ we estimate
	\begin{align*}
		&\E\Big[\exp\left(-\frac1{\gamma^2}\{\mu^{\gamma h}(\varphi)-\scalar{\varphi}{1}-\scalar{\gamma h}{\varphi}\}\right)1_{\mathcal B^C}\Big]\\
		=&\E\Big[\exp\left(-\frac1{\gamma^2}\{\mu^{\gamma h_{\ell}}(\varphi)-\scalar{\varphi}{1}-\scalar{\gamma h_{\ell}}{1}\}\right)\exp\left( \frac1\gamma\scalar{h-h_{\ell}}{\varphi} \right)\exp\left(\frac1{\gamma^2}(\mu_{\ell}-\mu) \right)1_{\mathcal B^C}\Big]\\
		\leq &e^{2\gamma}\E\Big[\exp\left(-\frac1{\gamma^2}\{\mu^{\gamma h_{\ell}}(\varphi)-\scalar{\varphi}{1}-\scalar{\gamma h_{\ell}}{\varphi}\}\right)\Big],
	\end{align*}
	which is uniformly bounded in $\gamma$ by virtue of Step 1. 
	
	It remains to consider the expectation on the set $\mathcal B$
	\begin{align*}
		\E\Big[\exp\left(-\frac1{\gamma^2}\{\mu^{\gamma h}(\varphi)-\scalar{\varphi}{1}-\scalar{\gamma h}{\varphi}\}\right)1_{\mathcal B}\Big].
	\end{align*} 
	Since $\mu^{\gamma h}(\varphi)$ is positive, we have
	\begin{align*}
		&\E\left[\exp\left(-\frac1{\gamma^2}\{\mu^{\gamma h}(\varphi)-\scalar{\varphi}{1}-\scalar{\gamma h}{\varphi}\}\right)1_{\mathcal B}\right]\leq e^{\gamma^{-2}}
		\E\left[\exp\left(-\frac1{\gamma}\scalar{h}{\varphi}\right)1_{\mathcal B}\right]\\
		&\leq e^{C\gamma^{-2}}\sqrt{\P[\mathcal B]},
	\end{align*} 
	where we used Cauchy Schwarz in the last inequality. This expectation can be bounded by virtue of Lemma \ref{lemma: set B} below in terms of
	\begin{align*}
		\E\left[\exp\left(-\frac1{\gamma^2}\{\mu^{\gamma h}(\varphi)-\scalar{\varphi}{1}-\scalar{\gamma h}{\varphi}\}\right)1_{\mathcal B}\right]\leq \exp(C\gamma^{-2}+c\gamma^{-2}\log \gamma),
	\end{align*}
	where ${\ell}(\gamma)=\lceil e^{\gamma^{-1/6}}\rceil$. Since the last expression is uniformly bounded for small $\gamma$, this completes the proof.
\end{proof}

\begin{lemma}\label{lemma: set B}
	Let $h$ and $(h_{\ell})_{\ell\in\N}$ be given by \eqref{eq: white noise trunc}. Then
	\begin{align*}
		\mathcal B^*_{\gamma,l}=\{\scalar{h-h_{\ell}}{\varphi}\geq \gamma^2\}\cup \{\mu^*_{\ell}-\mu^*\geq \gamma^3\},
	\end{align*}
	where 
	\begin{align*}
		\mu^*_{\ell}=\mu^{\gamma h_{\ell}}(\varphi),\quad \mu^*=\mu^{\gamma h}(\varphi).
	\end{align*}
	Then there exists a constant $c>0$ such that for all $0<\gamma\leq 1$ 
	\begin{align*}
		\P[\mathcal B^*_{\gamma,\ell(\gamma)}]\leq \exp(c\gamma^{-2}\log \gamma),
	\end{align*}
	where ${\ell}(\gamma)=\lceil e^{\gamma^{-1/6}}\rceil$.
	
\end{lemma}
\begin{proof}
	All we need to do is to estimate $\P[\mathcal B^*_{\gamma,\ell(\gamma)}]$. For this we consider the two events separately and write ${\ell}=\ell(\gamma)$ for brevity. First, since $\scalar{h-h_{\ell}}{\varphi}$ has a centered Gaussian distribution with variance of order ${\ell}^{-1}$
	by virtue of \eqref{eq: variance} we find   
	\begin{align}\label{eq: b1}
		\P[\scalar{h-h_{\ell}}{\varphi}\geq \gamma^2]\leq \exp(-c\gamma^4\ell)
	\end{align}
	by standard Gaussian estimates and since $\gamma^2 \ell$ is large.
	
	Second, we introduce
	\begin{align*}
		\phi(t)=\E^\ell[e^{t(\mu^*_{\ell}-\mu^*)}]=\E[e^{t(\mu^*_{\ell}-\mu^*)}|\mathcal F_{\ell}],
	\end{align*}
	where $\mathcal F_{\ell}$ is the sigma-algebra generated by $h_u$, $u\leq \ell$.
	We write 
	\begin{align*}
		\mu^*_{\ell}-\mu^*=-\int_M\varphi(y)\wick{e^{\gamma h_{\ell}(y)}}(\wick{e^{\gamma Y_{\ell}(y)}}-1)\, d\vol(y), 
	\end{align*}
	where 
	\begin{align*}
		Y_{\ell}=h-h_{\ell}.
	\end{align*}
	Then
	\begin{align*}
		\phi'(t)=&\E^\ell[(\mu^*_{\ell}-\mu^*)e^{t(\mu^*_{\ell}-\mu^*)}]\\
		=&\int \varphi(x)\wick{e^{\gamma h_{\ell}(x)}}\E^\ell[e^{t(\mu^*_{\ell}-\mu^*)}-e^{t(\mu^*_{\ell}-\mu^*)}\wick{e^{\gamma Y_{\ell}}}]\, d\vol(x).
	\end{align*}
	Applying \eqref{e:liouville-campbell} to the GMC associated to $Y_{\ell}$ with correlation function $\bar G_{\ell}(x,y)$,
	we find
	\begin{align}\label{eq: phi}
		\phi'(t)=\int \varphi(x)\wick{e^{\gamma h_{\ell}(x)}}\E^{\ell}[e^{t(\mu^*_{\ell}-\mu^*)}-e^{t(\mu^*_{\ell}-\tilde \mu_{\ell}^x)}]\, d\vol(x),
	\end{align}
	where 
	\begin{align*}
		\tilde \mu_{\ell}^x=\int e^{\gamma^2\bar G_{\ell}(x,y)}\wick{e^{\gamma h(y)}}\, d\vol(y).
	\end{align*}
	
	Note that, since $e^{-t \mu^*}$ and 
	\begin{align*}
		e^{t(\mu^*-\tilde \mu_{\ell}^x)}=\exp\left(-t\int (e^{\gamma^2\bar G_{\ell}(x,y)}-1)\wick{e^{\gamma h(y)}}\, d\vol(y)\right)
	\end{align*}
	are decreasing functions in $h-h_{\ell}$, the Gaussian FKG inequality \cite[Lemma 2.1]{duminil2023existence} gives us
	\begin{align}\label{eq: fkg}
		\E^\ell[e^{-t\tilde\mu_{\ell}^x}]
		\ge 
		\Bigl(1-t\!\int (e^{\gamma^2\bar G_{\ell}(x,y)}-1)\wick{e^{\gamma h_{\ell}(y)}}\,d\vol(y)\Bigr)
		\E^{\ell}[e^{-t\mu^*}].
	\end{align}
	
	More precisely, we apply the Gaussian FKG inequality to $h_m-h_{\ell}$, which is continuous and has nonnegative correlation by \eqref{eq: pos corr}. Then, Lemma 2.1 in \cite{duminil2023existence} provides
	\begin{align*}
		\E^\ell[e^{-t\tilde \mu_{m,l}^x}]\geq& \E^\ell[e^{t(\mu^*_m-\tilde \mu_{m,\ell}^x)}]\,\E^\ell[e^{-t\mu^*_m}]
	\end{align*}
	for all $m>\ell$, where
	\begin{align*}
		\tilde \mu_{m,\ell}^x=\int e^{\gamma^2\bar G_{m,\ell}(x,y)}\wick{e^{\gamma h_m(y)}}\, d\vol(y).
	\end{align*}
	and $\bar G_{m,\ell}(x,y)\geq0$ is the correlation function of $h_m-h_{\ell}$.
	Further, using the inequality $e^u\geq 1+u$,
	\begin{align*}
		\E^{\ell}[e^{t(\mu^*_m-\tilde \mu_{m,l}^x)}]\E^{\ell}[e^{-t\mu^*_m}]
		\geq &\E^{\ell}[1+t(\mu^*_m-\tilde \mu_{m,l}^x)]\E^{\ell}[e^{-t\mu^*_m}]\\
		=&(1-t\int[e^{\gamma^2\bar G_{m,l}(x,y)}-1]\wick{e^{\gamma h_{\ell}}}\, d\vol(y))\E^{\ell}[e^{-t\mu^*_m}].
	\end{align*}

	Letting $m\to\infty$, since $0\le e^{-t\mu^*_m}\le 1$ and by Theorem \ref{t:WhiteNoiseApproxGMC}
	\[
	\mu^*_m \to \mu^*, \quad \tilde\mu_{m,\ell}^x \to \tilde\mu_{\ell}^x
	\]
	in probability,
	\[
	\E^{\ell}[e^{-t\mu^*_m}] \to \E^{\ell}[e^{-t\mu^*}],
	\qquad 
	\E^{\ell}[e^{-t\tilde\mu_{m,\ell}^x}] \to \E^{\ell}[e^{-t\tilde\mu_{\ell}^x}].
	\]
	Moreover, $\bar G_{m,l}(x,y)\to \bar G_{\ell}(x,y)$, so by Fatou's lemma
	\[
	\int (e^{\gamma^2\bar G_{m,l}(x,y)}-1)\wick{e^{\gamma h_{\ell}(y)}}\,d\vol(y)
	\to 
	\int (e^{\gamma^2\bar G_{\ell}(x,y)}-1)\wick{e^{\gamma h_{\ell}(y)}}\,d\vol(y).
	\]
	Passing to the limit in the inequality therefore produces \eqref{eq: fkg}.

	Together with \eqref{eq: phi} we find after using $\wick{e^{\gamma h_{\ell}}}\leq e^{\gamma h_{\ell}}$
	\begin{align*}
		\phi'(t)\leq t\phi(t)\iint[e^{\gamma^2\bar G_{\ell}(x,y)}-1]e^{\gamma h_{\ell}(y)+\gamma h_{\ell}(x)}\, d\vol (x)\, d\vol (y).
	\end{align*}
	By Gronwall's inequality we deduce that
	\begin{align}\label{eq: phi2}
		\phi(t)\leq e^{t^2Z_{\ell}},
	\end{align}
	where 
	\begin{align*}
		Z_{\ell}=\iint[e^{\gamma^2\bar G_{\ell}(x,y)}-1]e^{\gamma h_{\ell}(y)+\gamma h_{\ell}(x)}\, d\vol (x)\, d\vol (y).
	\end{align*}
	Now we need to control $Z_{\ell}$. 
	By \eqref{eq: white-noise} there exists a constant $C>0$ such that for all $x\in M$ and $\gamma>0$
	\begin{align*}
		\int [e^{\gamma^2\bar G_{\ell}(x,y)}-1]\, d\vol(y)\leq C\gamma ^2 \ell^{-1}.
	\end{align*}

	Hence, after using the symmetry and
	\begin{align*}
		e^{\gamma h_{\ell}(y)+\gamma h_{\ell}(x)}\leq \frac12(e^{2\gamma h_{\ell}(y)}+e^{2\gamma h_{\ell}(x)})
	\end{align*}
	we obtain
	\begin{align*}
		Z_{\ell}\leq C\gamma ^2 \ell^{-1}\int e^{2\gamma h_{\ell}(x)}\, d\vol(x)
	\end{align*}
	After using standard Gaussian tail bounds we estimate
	\begin{align*}
		\E[\int e^{2\gamma h_{\ell}(x)}\, d\vol(x)]\leq &\sqrt{\ell}+\E[\int e^{2\gamma h_{\ell}(x)}1_{\{\gamma h_{\ell}(x)\leq \log(\ell/4)\}}\,d\vol(x)]\\
		\leq &\sqrt{\ell}+e^{-c\gamma^{-2}\log \ell}
	\end{align*}
	and hence 
	\begin{align*}
		\E[Z_{\ell}]\leq C\gamma^2\ell^{-1}\E[\int e^{2\gamma h_{\ell}(x)}\, d\vol(x)].
	\end{align*}
	Then, applying Markov's inequality to a sufficiently big power $Z_{\ell}^p$ and eventually using Kahane's convexity principle we obtain 
	\begin{align}\label{eq: condition2}
		\P[Z_{\ell}\geq \ell^{-1/2}]\leq e^{-c\gamma^{-2}\log \ell}.
	\end{align}
	
	Applying Chernoff's bound to the conditional expectation and estimate \eqref{eq: phi2}
	we find
	\begin{align*}
		\P[\mu^*_{\ell}-\mu^*\geq \gamma^3|\mathcal F_{\ell}]\leq e^{-t\gamma^3}e^{t^2Z_{\ell}}
	\end{align*}
	and consequently on the event $Z_{\ell}\leq  \ell^{-1/2}$ we get 
	\begin{align*}
		\P[\mu^*_{\ell}-\mu^*\geq \gamma^3|Z_{\ell}\leq  \ell^{-1/2}]\leq e^{-t\gamma^3}e^{t^2 \ell^{-1/2}}.
	\end{align*}
	Inserting $t=\ell^{1/4}$ we obtain
	\begin{align}\label{eq: condition}
		\P[\mu^*_{\ell}-\mu^*\geq \gamma^3|Z_{\ell}\leq  \ell^{-1/2}]\leq e^{-\ell^{1/4}\gamma^3}.
	\end{align}
	Finally, applying \eqref{eq: condition2} and \eqref{eq: condition}
	\begin{align}\label{eq: b2}
		\P[\mu^*_{\ell}-\mu^*\geq \gamma^3]\leq e^{-c\gamma^{-2}\log \ell}+e^{-\ell^{1/4}\gamma^3}.
	\end{align}
	This finishes the proof, since inserting ${\ell}=\ell(\gamma)$ we find that \eqref{eq: b1} and \eqref{eq: b2} are both dominated by $\exp(c\gamma^{-2}\log\gamma)$ for some constant $c>0$.
\end{proof}

\subsection{Conclusion of the argument}

Now we are ready to prove our main result.
\begin{proof}[Proof of Theorem \ref{thm: main}]
	$(i)$	We start with proving that as $\gamma\to0$, the law $\Q^{\mathrm{sc}}_{\kappa,\gamma}$ of the field $\frac{\gamma}n h$ with respect to $\Q_{\kappa,\gamma}$ converges weakly on $H^{-1}(M)$ towards $\delta_{u_\kappa}$. Note that by Lemma \ref{lemma: Qkappa} and since $u_\kappa=u_0+\kappa$, it is enough to show that $\Q_{0,\gamma}$ converges weakly towards $\delta_{u_0}$. Hence we concentrate in the following on the case $\kappa=0$ and we denote
	$$\Q_\gamma=\Q_{0,\gamma}, \qquad u=u_0, \qquad \Lambda=\Lambda_0.$$
	
	With a similar computation as in Proposition \ref{prop: pf} one checks that for any bounded continuous functional on $ H^{-1}$ 
	\begin{equation}\begin{aligned}\label{eq: girsanov}
			&\E\left[\Phi\left(\frac{\gamma}{n} h\right)\exp\left(-\frac{na_n}\gamma\scalar{h}{Q}-\frac{na_n\Lambda}{\gamma^2}\mu^{\gamma h}(M)\right)\right]\\
			&=e^{\gamma^{-2}F(\Lambda)}\E\left[\Phi\left(\frac{\gamma}n h+u\right)\exp\left(-\frac{na_n\Lambda}{\gamma^2}\scalar{e^{nu}}{\mu^{\gamma h}-1-\gamma h}\right)\right].
		\end{aligned}
	\end{equation}
	Since $\frac{\gamma}n h+u$ converges towards $u$ as $\gamma\to0$ in probability we have with Proposition \ref{prop: convergence} that 
	\begin{align*}
		\E\left[\Phi\left(\frac\gamma n h\right)\exp\left(-\frac{na_n}\gamma\scalar{h}{Q}-\frac{na_n\Lambda}{\gamma^2}\mu^{\gamma h}(M)\right)\right]\\
		\sim e^{\gamma^{-2}F(\Lambda)}\Phi(u)\E\left[\exp(-\frac{na_n\Lambda}{2}\scalar{e^{nu}}{\wick{h^2}})\right].
	\end{align*}
	Consequently through normalization as $\gamma\to0$
	\begin{align*}
		(Z_\gamma)^{-1}\E\left[\Phi\left(\frac{\gamma}n h\right)\exp\left(-\frac{na_n}\gamma\scalar{h}{Q}-\frac{na_n\Lambda}{\gamma^2}\mu^{\gamma h}(M)\right)\right]
		\to\Phi(u),
	\end{align*}
	where we used Proposition \ref{prop: pf}. This shows that $\frac\gamma n h\to u$ in law with respect to $\Q_\gamma$.

	$(ii)$	Next, we show that as $\gamma\to0$, the law of
	the field $h-\frac n\gamma u$ with respect to $\Q_\gamma$ converges weakly on $H^{-1}(M)$ towards the law of an augmented massive field $\tilde h$ with mass $na_n\Lambda$ in the metric $\hat g=e^{2u}g$.

	We define
	\begin{align*}
		h^u=h-\frac{n}\gamma u.
	\end{align*}
	
	Then by \eqref{eq: girsanov} and the transformation rule we have
	\begin{align*}
		&\E\left[\Phi(h^u)\exp\left(-\frac{na_n}\gamma\scalar{h}{Q}-\frac{na_n\Lambda}{\gamma^2}\mu^{\gamma h}(M)\right)\right]\\
		&=e^{\gamma^{-2}F(\Lambda)}\E\left[\Phi(h^u+\frac{n}\gamma u)\exp\left(-\frac{na_n\Lambda}{\gamma^2}\scalar{e^{nu}}{\mu^{\gamma h}-1-\gamma h}\right)\right]\\
		&=e^{\gamma^{-2}F(\Lambda)}\E\left[\Phi(h)\exp\left(-\frac{na_n\Lambda}{\gamma^2}\scalar{e^{nu}}{\mu^{\gamma h}-1-\gamma h}\right)\right]
	\end{align*}
	Hence through normalization by Proposition \ref{prop: convergence} and Proposition \ref{prop: pf} we arrive at
	\begin{align*}
		&\lim_{\gamma\to0}(Z_\gamma)^{-1}\E\left[\Phi(h^u)\exp\left(-\frac{na_n}\gamma\scalar{h}{Q}-\frac{na_n\Lambda}{\gamma^2}\mu^{\gamma h}(M)\right)\right]\\
		&=(\tilde Z_u)^{-1}\E\left[\Phi(h)\exp\left(-\frac{na_n\Lambda}{2}\scalar{e^{nu}}{\wick{h^2}}\right)\right],
	\end{align*}
	where 
	\begin{align*}
		\tilde Z_u=\E\left[\exp\left(-\frac{na_n\Lambda}{2}\scalar{e^{nu}}{\wick{h^2}}\right)\right].
	\end{align*}

	According to Lemma \ref{lem: wickfinite} 
	we know that under the measure 
	\begin{align*}
		\tilde\P_u=\tilde Z_u^{-1}\exp\left(-\frac{na_n\Lambda}{2}\scalar{e^{nu}}{\wick{h^2}}\right)\P
	\end{align*}
	the field $h$ is an augmented 
	massive field with mass $na_n\Lambda$ under the metric $\hat g$.
\end{proof}

\bibliographystyle{plain}
\bibliography{liouville_semiclassical}

\end{document}